\documentclass[a4paper,12pt]{article}
\usepackage[text={6.5in,9.5in},centering]{geometry}
\usepackage{graphicx,url,subfig}
\RequirePackage[OT1]{fontenc}
\RequirePackage{amsthm,amsmath,amssymb,amscd}
\RequirePackage[colorlinks,citecolor=blue,urlcolor=blue]{hyperref}
\usepackage{enumerate}
\usepackage{datetime}
\usepackage{xfrac}
\usepackage[normalem]{ulem} 
\usepackage{soul} 
\makeatother
\numberwithin{equation}{section}
\allowdisplaybreaks
\usepackage{dsfont}
\usepackage{mathtools}
\usepackage{xcolor}

\theoremstyle{plain}
\newtheorem{theorem}{Theorem}[section]

\newtheorem{lemma}[theorem]{Lemma}
\newtheorem{proposition}[theorem]{Proposition}

\theoremstyle{definition}
\newtheorem{definition}[theorem]{Definition}

\newtheorem{remark}[theorem]{Remark}

\newcommand{\E}{\mathbb{E}}
\renewcommand{\P}{\mathbb{P}}

\newcommand{\ud}{\ensuremath{\mathrm{d} }}

\newcommand{\Norm}[1]{\left\|  #1 \right\|}

\newcommand{\calB}{\mathcal{B}}

\newcommand{\calF}{\mathcal{F}}

\newcommand{\calN}{\mathcal{N}}

\newcommand{\R}{\mathbb{R}}

\DeclareMathOperator{\Lip}{\mathit{L}}

\usepackage{amsrefs}

\BibSpec{article}{%
    +{}  {\PrintAuthors}                {author}
    +{,} { \textit}                     {title}
    +{.} { }                            {part}
    +{:} { \textit}                     {subtitle}
    +{,} { \PrintContributions}         {contribution}
    +{.} { \PrintPartials}              {partial}
    +{,} { }                            {journal}
    +{}  { \textbf}                     {volume}
    +{}  { \PrintDatePV}                {date}
    +{,} { \issuetext}                  {number}
    +{,} { \eprintpages}                {pages}
    +{,} { }                            {status}
    +{,} { \url}                        {url}
    +{,} { \PrintDOI}                   {doi}
    +{,} { available at \eprint}        {eprint}
    +{}  { \parenthesize}               {language}
    +{}  { \PrintTranslation}           {translation}
    +{;} { \PrintReprint}               {reprint}
    +{.} { }                            {note}
    +{.} {}                             {transition}
    +{}  {\SentenceSpace \PrintReviews} {review}
}

\title{A stochastic heat equation with non-locally \\ Lipschitz coefficients}

\author{Le Chen, Jingyu Huang, and Wenxuan Tao}

\begin{document}
\maketitle
\begin{abstract}
  We consider the \emph{stochastic heat equation} (SHE) on the torus
  $\mathbb{T}=[0,1]$, driven by space-time white noise $\dot W$, with an initial
  condition $u_0$ that is nonnegative and not identically zero:
  \begin{equation*}
    \frac{\partial u}{\partial t} = \tfrac{1}{2}\frac{\partial^2 u}{\partial x^2} + b(u) + \sigma(u)\dot{W}\,.
  \end{equation*}
  The drift $b$ and diffusion coefficient $\sigma$ are Lipschitz continuous away
  from zero, although their Lipschitz constants may blow up as the argument
  approaches zero. We establish the existence of a unique global mild solution
  that remains strictly positive. Examples include \(b(u)=u|\log u|^{A_1}\) and
  \(\sigma(u)=u|\log u|^{A_2}\) with \(A_1\in(0,1)\) and \(A_2\in(0,1/4)\).

  \bigskip

    

  \noindent{\it \noindent }

\end{abstract}

\setlength{\parindent}{1.5em}


\section{Introduction}

In this paper we study the stochastic heat equation (SHE) on the one-dimensional
torus $\mathbb{T}=[0,1]$ (i.e., the unit interval with periodic boundary
conditions),
\begin{equation}\label{E:SHE}
  \frac{\partial u}{\partial t} = \tfrac{1}{2}\frac{\partial^2 u}{\partial x^2} + b(u) + \sigma(u)\dot{W}\,,
\end{equation}
subject to a bounded initial condition $u_0(x)$, driven by space-time white
noise $\dot W$. We assume that both $b$ and $\sigma$ are locally Lipschitz
continuous away from zero, although their Lipschitz constants may blow up at
both zero and infinity. The solution to~\eqref{E:SHE} is understood in the mild
form:
\begin{equation}
\begin{aligned}\label{E:mild}
  u(t,x) = \int_0^1 G_t(x-y) u_0(y) \ud y
  + & \int_0^t \int_0^1 G_{t-s}(x-y)b(u(s,y))\ud y \ud s \\
  + & \int_0^t \int_0^1 G_{t-s}(x-y) \sigma(u(s,y)) W(\ud s,\ud y )\,,
\end{aligned}
\end{equation}
where the stochastic integral is the It\^o--Walsh
integral~\cite{walsh:86:introduction} and $G$ is the heat kernel on the torus:
\begin{equation}\label{E:HeatKernel}
  G_t(x) = \sum_{k \in \mathbb{Z}} \frac{1}{\sqrt{2\pi t}} e^{-\frac{(x-k)^2}{2t}}\,.
\end{equation}

It is well known that superlinearity may cause the solution to parabolic partial
differential equations to blow up (see, for example,
\cite{galaktionov.vazquez:02:problem}). When it comes to stochastic partial
differential equations (SPDEs); there has been extensive work on blowup
phenomena see, e.g., \cite{mueller:97:long}, \cite{mueller:98:long-time},
\cite{mueller:00:critical}, \cite{mueller.sowers:93:blowup}.
Salins~\cite{salins:25:solutions} focuses on the case without a drift term.
Mytnik \& Perkins~\cite{mytnik.perkins:11:pathwise} study the stochastic heat
equation with globally Lipschitz drift and multiplicative noise, \ but with a
non-Lipschitz diffusion coefficient.

In the case that there is a non-Lipschitz drift term, Fern\'andez Bonder and
Groisman~\cite{fernandez-bonder.groisman:09:time-space} demonstrated the
importance of the so-called \emph{Osgood condition} in SPDEs. The Osgood
condition~\cite{osgood:98:beweis}, originating in the theory of ordinary
differential equations (ODEs), generalizes the Lipschitz condition to ensure the
uniqueness. More precisely, a function $h:\R\to \R$ satisfies the \emph{infinite
Osgood condition} if
\begin{align}\label{E:Osgood}
    \int_c^\infty\frac{1}{h(z)}\ud z=\infty\,.
\end{align}
The main result in Fern\'andez Bonder and
Groisman~\cite{fernandez-bonder.groisman:09:time-space} is that the solution
to~\eqref{E:SHE} blows up in finite time almost surely if $b$ does not satisfy
the infinite Osgood condition~\eqref{E:Osgood} and $\sigma$ is constant. Foondun
and Nualart~\cite{foondun.nualart:21:osgood} later proved that if $\sigma$ is
constant, the Osgood condition is both necessary and sufficient for the
existence of a global solution. This result holds regardless of whether the
domain is bounded or unbounded, and remains valid under colored noise (with
Riesz kernel-type spatial correlation) as well as in settings involving the
fractional Laplacian.

Fern\'andez Bonder \& Groisman~\cite{fernandez-bonder.groisman:09:time-space}
also motivated research into the case where $\sigma$ is not constant. Using a
stopping-time argument, Dalang, Khoshnevisan and
Zhang~\cite{dalang.khoshnevisan.ea:19:global} established that~\eqref{E:SHE}
admits a unique global mild solution provided that $b$ and $\sigma$ are locally
Lipschitz and satisfy the following growth conditions at infinity:
\begin{align*}
    \left|b(z)\right|=O\left(|z|\log |z|\right)\quad\text{and}\quad\left|\sigma(z)\right|=o\left(|z|\left(\log |z|\right)^{1/4}\right)\,,
    \quad \text{as $|z|\to \infty$.}
\end{align*}
Following the work of~\cite{dalang.khoshnevisan.ea:19:global}, there has
recently been significant research on SPDEs with non-Lipschitz drift terms; see,
for example, \cite{salins:21:existence}, \cite{salins:22:global},
\cite{salins:22:global*1}, as well as~\cite{millet.sanz-sole:21:global} for
results concerning stochastic wave equations. The result
in~\cite{dalang.khoshnevisan.ea:19:global} was subsequently extended by Chen and
Huang~\cite{chen.huang:23:superlinear} to the setting of $\R^d$ and to noise
that is white in time and colored in space. Moreover, the condition on $b$ has
recently been further relaxed to the infinite Osgood condition~\eqref{E:Osgood}
in~\cite{chen.foondun.ea:23:global}. We would like to mention that
Kotelenez~\cite[Theorem 3.4]{kotelenez:92:comparison} established the existence
of a global mild solution without the Lipschitz condition. Instead, one-sided
Lipschitz condition (see A.3.iii therein) of the drift term $b$ was assumed.

\bigskip

All the aforementioned literature on global existence and uniqueness of the
solution considers the effects of superlinearity of $b$ or $\sigma$ only at
infinity; that is, the Lipschitz coefficients are allowed to blow up at
infinity as the argument tends to infinity, while both functions are assumed to
be locally Lipschitz, meaning they exhibit regular behavior near zero. In this
paper, we are going to consider the \emph{opposite scenario}, namely, when the
Lipschitz coefficients of the drift term $b$ and the diffusion coefficient
$\sigma$ may blow up at $0$. A typical example is $b(z) = z \log z$ for $z$
close to $0$. Such a logarithmic correction naturally appears in relativistic
and wave mechanical contexts, for instance, in~\cite[equations
(3.9)]{bialynicki-birula.mycielski:76:nonlinear} and~\cite[equation
(1)]{rosen.gerald:69:Dilatation}. To the best of our knowledge, Shang and
Zhang~\cite{shang.zhang:22:stochastic} were the first to establish the
existence of a global solution to the stochastic heat equation under the
assumption that $b(z) = z|\log z|$ for all $z \in \mathbb{R}$, which lacks
local Lipschitz continuity at $z=0$. However, they still required $\sigma$ to
be locally Lipschitz around zero and assumed that the driving noise is a
standard Brownian motion. Pan, Shang et al.~\cite{pan.shang.ea:25:large} later
established a large deviation principle under the same setup. In a recent work,
Kavin \& Majee~\cite{Kavin.Majee:25:levy} studied SPDEs driven by Lévy noise,
incorporating a logarithmic correction in the drift term.

In the case of SPDE driven by space-time white noise, recently, Han, Kim and
Yi~\cite{han.kim.ea:24:on} established the well-posedness, strict positivity and
compact support property of the global mild solution to the stochastic heat
equation driven by space-time white noise on $\R$ under the assumption $b \equiv
0$ and
\begin{align}\label{E:HanKimYiAssumption}
    0< \sigma(z)\le z|\log z|^{\alpha}\quad \text{for some $\alpha \in (0,1/4)$, for $z$ near $0$}\,.
\end{align}

Their approach relies on a classical technique involving the subtraction of two
solutions followed by an application of Gronwall's lemma. A key ingredient in
their analysis is the additional assumption that $\sigma(z)/z$ is monotone
decreasing, which plays a crucial role in their proof.

In this paper, we consider SHE~\eqref{E:SHE} which includes a nonvanishing drift
term $b$, and we assume that, as $z$ approaches {$0^+$},
\begin{equation}\label{E:OurAssumption}
  \begin{dcases}
    |b(z)|      = O\left(z|\log z|^{A_1}\right) & \text{for some $A_1\in (0,1)$,} \\[0.5em]
    |\sigma(z)| = O\left(z|\log z|^{A_2}\right) & \text{for some $A_2\in (0,1/4)$.}
  \end{dcases}
\end{equation}
Our assumption on $\sigma$ is similar to condition~\eqref{E:HanKimYiAssumption}
in~\cite{han.kim.ea:24:on}, but we allow $\sigma$ to be in a wider class of
functions. Specifically, we do not require $\sigma(z)/z$ to be monotone
decreasing, nor do we impose any monotonicity or regularity conditions beyond
the stated growth bound. Regarding the assumption on $b(z)$, if the drift term
$b(z)$ stays positive as $z\to 0^+$, it actually helps the solution stay
positive. However, providing a rigorous justification for this statement under
minimal assumptions on the positive drift $b$ requires additional analysis,
which lies beyond the scope of the present paper. In this paper, we emphasize
that the non-positive drift case, where $b(z)$ may drive the solution towards
zero, is nontrivial. Examples of such non-positive drift include
\begin{align}\label{E:Examples-b}
  b(z) = - z \left(\log(1 /z)\right)^{A_1} \quad \text{and} \quad 
  b(z) = z \left(\log(1 /z)\right)^{A_1} \sin\left(1/z\right).
\end{align}

We establish the existence, uniqueness, and strict positivity of the global mild
solution by introducing a \emph{new methodology} based on stopping time
arguments, which can be summarized as follows. Given a positive initial
condition $u_0$, we modify $b(z)$ (and similarly $\sigma(z)$) for $z \in
(0,\epsilon)$ to obtain $b_{\epsilon}(z)$ (and $\sigma_{\epsilon}(z)$), so that
these functions become globally Lipschitz, thereby ensuring a global unique
solution $u_{\epsilon}$. Next, we construct a sequence of stopping times
$\tau_{\epsilon}$ such that, prior to $\tau_{\epsilon}$, the solution
$u_{\epsilon}$ remains greater than $\epsilon$. Consequently, before
$\tau_{\epsilon}$, $b_{\epsilon}(u_{\epsilon}) = b(u_{\epsilon})$ (and similarly
for $\sigma_{\epsilon}$), and we define the solution $u$ to coincide with
$u_{\epsilon}$ up to the stopping time $\tau_{\epsilon}$. Finally, we show that
the stopping times $\tau_{\epsilon}$ converge to any fixed time horizon $T$
almost surely as $\epsilon \to 0$, thereby obtaining the solution at any time;
see Theorem~\ref{T:SubCRT}. This procedure requires a careful analysis of the
probability that $u_{\epsilon}(t,x) < \epsilon$, which forms the main part of
Section~\ref{S:MainResults}. With a minor additional assumption
(see~\eqref{E:Liminf} below), this result can be partially extended to the case
$A_1 = 1$; see Theorem~\ref{T:Critical}. From this analysis, we also observe
that the blow-up of Lipschitz coefficients does not affect the moment bounds of
the solution. Consequently, we may apply the same method as
in~\cite{dalang.khoshnevisan.ea:19:global} to further accommodate the
superlinear growth of $b(z)$ and $\sigma(z)$ as $z\to \infty$; see
Theorem~\ref{T:Superlinear}.

The remainder of this paper is organized as follows. In
Section~\ref{S:Preliminary}, after introducing some preliminary concepts about
the noise structure and the definition of a mild solution, we prove some moment
estimates for the solution. Section~\ref{S:MainResults} consists of our main
result---Theorem~\ref{T:SubCRT}---along with its proof. In particular, this
section addresses the case $b(z) = O(z|\log z|^{A_1})$ for all $A_1 \in (0,1)$.
Subsequently, in Section~\ref{S:Extensions}, we extend our main result in two
directions: (1) we consider the case $A_1 = 1$ in Theorem~\ref{T:Critical} as a
limiting case as $A_1 \uparrow 1$, and (2) we extend our results to allow both
$b$ and $\sigma$ to exhibit superlinear growth at infinity in
Theorem~\ref{T:Superlinear}, as in~\cite{dalang.khoshnevisan.ea:19:global}.
Finally, Appendix~\ref{S:Appendix} collects several results that are used
extensively throughout the paper.

\paragraph{Notation} Throughout the paper, we will use $\Norm{\cdot}_p$ to
denote the $L^p(\Omega)$ norm and $\Norm{\cdot}_\infty$ the
$L^\infty(\mathbb{T})$ norm. We will use $C$ to denote a generic constant that
may change its value at each appearance. For $x, y \in \mathbb{T}$, $|x-y|$ is
understood as the distance on the torus. Let $L_b$ denote the growth constant of
$b$, which is defined as
\begin{align}\label{E:Lb}
  L_b \coloneqq \sup_{z>0} \frac{|b(z) - b(0)|}{z}\,.
\end{align}
Similarly, we use $L_\sigma$ to denote the growth constant of $\sigma$.

\paragraph{Acknowledgements} L.C. and J.H. thank Mickey Salins for valuable
discussions regarding conditions on the drift term $b(z)$. L.C. was partially supported by the NSF grant DMS-Probability (2023-2026, Award Number 2246850) and the Simons Foundation Travel Grant for Mathematicians (2022-2027, Award Number MPS-TSM-00959981).

\section{Preliminary estimates} \label{S:Preliminary}

To begin with, let $W \coloneqq \left\{ W_t(A), \:A \in
\calB\left(\mathbb{T}\right), t \ge 0 \right\}$ be a space-time white noise
defined on a complete probability space $(\Omega,\calF,\P)$, where
$\calB\left(\mathbb{T}\right)$ is the collection of Borel measurable sets. Let
\begin{align*}
  \calF_t \coloneqq \sigma\left(W_s(A),\:0\le s\le
  t,\:A\in\calB\left(\mathbb{T}\right)\right)\vee
  \calN,\quad t\ge 0,
\end{align*}
be the natural filtration of $W$ combined with the $\sigma$-field $\calN$
generated by all $\P$-null sets in $\calF$. Throughout the paper, we fix the
filtered probability space $\left\{\Omega,\calF,\{\calF_t,\:t\ge0\},\P\right\}$.
Under this setup, $W$ becomes a worthy martingale measure in the sense of
Walsh~\cite{walsh:86:introduction}, and $\iint_{[0,t]\times \mathbb{T}}X(s,y)
W(\ud s,\ud y)$ is well-defined in this reference for a suitable class of random
fields $\left\{X(s,y),\; (s,y)\in\R_+\times \mathbb{T}\right\}$. A rigorous
definition of the mild solution to~\eqref{E:SHE} is given below.

\begin{definition}\label{D:Solution}
  A process $u=\left(u(t,x),\:(t,x) \in (0,\infty) \times \mathbb{T}\right)$ is
  called a \emph{random field solution} to~\eqref{E:mild} if
  \begin{enumerate}[(1)]

   \item $u$ is adapted, i.e., for all $(t,x) \in (0,\infty) \times \mathbb{T}$,
     $u(t,x)$ is $\calF_t$-measurable;

  \item $u$ is jointly measurable with respect to $\calB\left( (0,\infty) \times
    \mathbb{T}\right) \times \calF$;

  \item for all $(t,x) \in (0,\infty) \times \mathbb{T}$,
    \begin{align*}
      \int_0^t\int_{\mathbb{T}} G_{t-s}^2(x-y) \E\left( \left|\sigma\left(u(s,y)\right)\right|^2 \right) \ud s \ud y < +\infty;
    \end{align*}

  \item For all $(t,x) \in (0,\infty) \times
    \mathbb{T}$, $u$ satisfies \eqref{E:mild} a.s.

  \end{enumerate}
\end{definition}


The drift term $b$ and the diffusion coefficient $\sigma$ appearing in both
Theorems~\ref{T:SubCRT} and~\ref{T:Critical} below satisfy the linear growth
condition. This condition is formalized in the following elementary lemma.
\begin{lemma}\label{L:LinearGrowth}
  Let $f: [0,\infty) \to \R$ be a continuous function such that, for some
  $\delta > 0$, $f$ is a Lipschitz function on $[\delta, \infty)$. Denote the
  corresponding Lipschitz constant by $L_{{f,\delta}}$. Then
  \begin{align}\label{E:LinearGrowth}
    | f(x) | \le L_{{f,\delta}}\: x + C_{{f,\delta}}  \quad \text{for all $x \ge 0$ with $C_{{f,\delta}} \coloneqq \sup_{y \in [0, \delta]} | f(y) | $}.
  \end{align}
\end{lemma}

Under the Lipschitz assumptions on $b$ and $\sigma$, together with the
boundedness assumption on the initial condition, standard fixed point arguments
can be used to establish the existence and uniqueness of the solution to
SHE~\eqref{E:SHE}. The following proposition provides a moment bound, with
particular attention to the precise dependence on the Lipschitz constants and
the parameter $p$.

\begin{proposition}\label{P:MomBdU}
  Consider SHE~\eqref{E:SHE} with bounded initial condition $u_0$.

  \noindent(1) Assume that both $b$ and $\sigma$ are globally Lipschitz with
  $b(0) = \sigma(0) = 0$. Let $u$ be the unique solution. Then for all $T > 0$
  and $p \ge 2$, $u$ satisfies the following moment bound:
  \begin{equation}\label{E:MomBdU-1}
    \sup_{t\in [0,T]} \sup_{x\in \mathbb{T}} \E \left(|u(t,x)|^p\right)
    \le 2^p\|u_0\|_{\infty}^p \exp \left( \left(4 L_b p + 2^{16} \pi^2 p^3 L_{\sigma}^4 \right) T \right)\,.
  \end{equation}

  \noindent(2) Assume that for some $\delta > 0$, both $b$ and $\sigma$ satisfy
  conditions in Lemma~\ref{L:LinearGrowth} with the same $\delta$. Then any
  solution $u$ satisfies the following a priori moment estimate similar
  to~\eqref{E:MomBdU-1}:
  \begin{equation}\label{E:MomBdU-2}
    \sup_{t\in [0,T]} \sup_{x\in \mathbb{T}} \E \left(|u(t,x)|^p\right)
    \le 2^p\left(\|u_0\|_{\infty} + \frac{C_{{b,\delta}}}{4 L_{b, \delta}} + \frac{C_{{\sigma,\delta}}}{4L_{\sigma, \delta}}\right)^p
    \exp \left( \left(4 L_{b, \delta}\, p + 2^{16} \pi^2 p^3 L_{\sigma, \delta}^4 \right) T \right)\,,
  \end{equation}
  for all $T > 0$ and $p \ge 2$, where the constants $C_{b, \delta}$ and $L_{b,
  \delta}$ (respectively, $C_{\sigma, \delta}$ and $L_{\sigma, \delta}$) are
  defined as in Lemma~\ref{L:LinearGrowth}, with $f$ replaced by $b$
  (respectively, $\sigma$).
\end{proposition}

\begin{proof}
  (1) The method of proof is standard, (e.g.~\cite{khoshnevisan:14:analysis}),
  we just sketch the key steps. Consider Picard iteration, let $$ u_1(t,x) =
  \int_0^1 G_t(x-y)u_0(y)\ud y\,, $$ and define $u_n$ iteratively using
  \eqref{E:mild} with the right hand side $u$ replaced by $u_{n-1}$,
  \begin{equation}
    \begin{aligned}\label{E:Picard}
        u_n(t,x) = \int_0^1 G_t(x-y) u_0(y) \ud y
         + & \int_0^t \int_0^1 G_{t-s}(x-y)b(u_{n-1}(s,y))\ud y \ud s \\
        + & \int_0^t \int_0^1 G_{t-s}(x-y) \sigma(u_{n-1}(s,y)) W(\ud s,\ud y )\,.
    \end{aligned}
  \end{equation}
  For $p \ge 2$, apply the $\Norm{\cdot}_p$ norm, along with the
  Burkholder-Davis-Gundy (see, e.g., the version in Remark~2.2
  of~\cite{foondun.khoshnevisan:09:intermittence}) and Minkowski inequalities
  to~\eqref{E:Picard} to see that
  \begin{align}\label{E_:BDG}\notag
    \Norm{u_n(t,x)}_p \le
    & \left|\int_0^1 G_t(x-y) u_0(y) \ud y \right| +  L_b\int_0^t \int_0^1 G_{t-s}(x-y) \Norm{u_{n-1}(s,y)}_p\ud y \ud s \\
    & + 2 \sqrt{p}L_{\sigma} \left(\int_0^t \int_0^1 G_{t-s}^2(x-y)  \Norm{u_{n-1}(s,y)}^2_p\ud y  \ud s\right)^{1/2}\,.
  \end{align}
  Consider the norm
  \begin{equation}\label{E:STN}
    \mathcal{N}_{\kappa,\: p} (u) \coloneqq
    \sup_{t \ge 0} \sup_{x\in \mathbb{T}} e^{-\kappa t} \Norm{u(t,x)}_p
    \quad \text{for ${\kappa > 0}$ and $p \ge 2$.}
  \end{equation}
  For $p \ge 2$ and ${\kappa > 0}$, apply the norm in~\eqref{E:STN} to both
  sides of the above inequality to obtain that
  \begin{align*}
    \mathcal{N}_{\kappa,\: p}(u_n)
    \le & \|u_0\|_{\infty} + L_b\int_0^t e^{-\kappa(t-s)} \; \mathcal{N}_{\kappa, \: p}(u_{n-1}) \ud s \\
        & + 2 \sqrt{p} L_{\sigma}\left(\int_0^t \int_0^1 e^{-2(t-s)\kappa} G_{t-s}^2(x-y)  \: \mathcal{N}_{\kappa, \: p}(u_{n-1})^2 \ud y  \ud s\right)^{1/2}\,.
  \end{align*}
  By the semigroup property of $G_t(x)$ we have that
  \begin{align*}
    \int_0^1 G_{t-s}^2(x-y) \ud y  = G_{2(t-s)} (0) \le \left(1+ \frac{\sqrt{2\pi}}{\sqrt{t-s}}\right)\,,
  \end{align*}
  where the inequality follows from (2.22)
  of~\cite{chen.ouyang.ea:23:parabolic}.

  Therefore,
  \begin{align}\label{E:STN-Contraction}
    \mathcal{N}_{\kappa,\: p}(u_n)
    \le \|u_0\|_\infty + \frac{1}{\kappa} L_b\: \mathcal{N}_{\kappa,\: p}(u_{n-1})
    + 2  \sqrt{p} L_{\sigma}\left({\frac{\pi}{\sqrt{\kappa}}} + \frac{1}{{2\kappa}}\right)^{1/2} \mathcal{N}_{\kappa,\: p}(u_{n-1})\,.
  \end{align}
  Now, the optimal condition on $\kappa$ such that~\eqref{E:STN-Contraction}
  becomes a contraction is that
  \begin{align}\label{E:OptimalKappa}
    \frac{L_b}{\kappa}
    + 2 \sqrt{p} L_\sigma \left( \frac{\pi}{\sqrt{\kappa}} + \frac{1}{2\kappa} \right)^{1/2}
    < 1.
  \end{align}
  It is clear that the left-hand side of \eqref{E:OptimalKappa} is monotonically
  decreasing in $\kappa$. Therefore, there exists an optimal value $\kappa_*
  \coloneqq \kappa_*(L_b, L_\sigma, p)$, defined as the smallest positive real
  number for which the inequality in \eqref{E:OptimalKappa} becomes an equality.
  In principle, $\kappa_*$ can be solved explicitly, which would require solving
  a quartic polynomial equation. However, the explicit expression is rather
  complicated and not particularly convenient for practical use. For our
  purpose, it suffices to find a sub-optimal value of $\kappa$ that satisfies
  the inequality~\eqref{E:OptimalKappa}. In particular, the following conditions
  will be sufficient for~\eqref{E:STN-Contraction} to become a contraction:
  \begin{align*}
    \frac{L_b}{\kappa} < \frac{1}{4}, \quad
    2 \sqrt{p} L_\sigma \left( \frac{\pi}{\sqrt{\kappa}} \right)^{1/2} < \frac{1}{8},
    \quad \text{and} \quad
    2 \sqrt{p} L_\sigma \left( \frac{1}{2\kappa} \right)^{1/2} < \frac{1}{8}.
  \end{align*}
  From these inequalities, we readily obtain
  \begin{align*}
    \kappa > 4 L_b, \quad
    \kappa > 2^{16} \pi^2 p^2 L_{\sigma}^4, \quad
    \text{and} \quad
    \kappa > {2^7p L_\sigma^2}.
  \end{align*}
  Hence, we may take $\kappa = 4 L_b + 2^{16} \pi^2 p^2 L_{\sigma}^4$. Finally,
  iterate the inequality~\eqref{E:STN-Contraction} for $\mathcal{N}_{\kappa, \:
  p}(u_n)$ we obtain that
  \begin{equation}
    \sup_{n}\mathcal{N}_{\kappa, \: p}(u_n)< \infty\,.
  \end{equation}
  Standard iteration shows that $u_n$ converges to $u$ in the
  $\mathcal{N}_{\kappa, \: p}$ norm, and thus we obtain the bound in part (1) of
  Proposition~\ref{P:MomBdU}. \bigskip

  \noindent(2) In this case, we apply the same argument as in part (1) but
  replace the linear growth bound $\left|b(z)\right|\le L_bz$ (respectively,
  $\left|\sigma(z)\right|\le L_\sigma z$) with the linear growth estimate
  provided by Lemma~\ref{L:LinearGrowth}.
\end{proof}

In the next proposition, we study the Hölder regularity of the solution. As will
be demonstrated in the proof of Theorem~\ref{T:SubCRT}, it is necessary to
explicitly keep track of the dependence on the Lipschitz coefficients
$L_{\sigma}$ and $L_b$.

\begin{proposition}\label{P:Holder}
  Let $u$ be a solution to the stochastic heat equation (SHE)~\eqref{E:SHE} with
  initial condition $u_0$, where $u_0$ is Hölder continuous with exponent
  $\gamma > 0$. Define
  \begin{align*}
    |u_0|_{\gamma} \coloneqq \sup_{x, x' \in \mathbb{T}} \frac{|u_0(x) - u_0(x')|}{|x - x'|^{\gamma}}.
  \end{align*}
  {Fix an arbitrary $p \ge 2$}. We have the
  following two cases: \bigskip

  \noindent(1) Assume that both $b$ and $\sigma$ are globally Lipschitz. For all
  $t, t'\in [0, {\infty})$ with {$t < t'$}
  {and $t'-t<1$}, $x, x'\in \mathbb{T}$ and for all $\beta
  \in (0, 1/2\wedge \gamma)$, there exists a constant $C_\beta$ that only
  depends on $\beta$, such that
  \begin{align}\label{E:SpaceInc}
    \begin{aligned}
    \Norm{u(t,x)-u(t,x')}_p
    \le {C_\beta} |x-x'|^{\beta} \Bigl\{
     & |u_0|_{\gamma} +  |b(0)|\,t^{1-\beta/2}                             \\
     & + \sqrt{p}\,|\sigma(0)|\left(t^{1/4-\beta/2}+t^{1/2-\beta/2}\right) \\
     & +  L_b M e^{C H t} H^{\beta/2-1}                                    \\
     & +  \sqrt{p}L_\sigma M e^{C H t} \left(H^{\beta/2-1/4} +  H^{\beta/2-1/2}\right)
    \Bigr\}
    \end{aligned}
  \end{align}
  and
  \begin{align}\label{E:TimeInc}
    \begin{aligned}
    \Norm{u(t,x)-u(t',x)}_{p}
    \le {C_\beta} (t'-t)^{\beta/2}\Bigl\{
      & |u_0|_{\gamma} + |b(0)|\left(1 + t^{1-\beta/2}\right)                    \\
      & + \sqrt{p} |\sigma(0)| \left( 1 + t^{1/4-\beta/2}+t^{1/2-\beta/2}\right) \\
      & + L_b M e^{C H t'} H^{\beta/2-1}                                         \\
      & + \sqrt{p} L_{\sigma} M  e^{C H t'} \left(H^{\beta/2-1/4}+H^{\beta/2-1/2}\right)  \Bigr\},
    \end{aligned}
  \end{align}
  where the constant $C$ appearing in the exponents is a generic constant and
  the constants $H$ and $M$ are defined as follows:
  \begin{align}\label{E:H&M}
    \begin{aligned}
      H & \coloneqq H\left(p, L_b, L_\sigma\right) = L_b+p^2L_{\sigma}^4, \quad \text{and} \\
      M & \coloneqq M(u_0, b,\sigma) = \|u_0\|_{\infty}+ \frac{|b(0)|}{L_b} + \frac{|\sigma(0)|}{L_{\sigma}}\,.
    \end{aligned}
  \end{align}

  \noindent(2) Assume that for some $\delta > 0$, both $b$ and $\sigma$ satisfy
  conditions in Lemma~\ref{L:LinearGrowth} with the same $\delta$. Let the
  constants $C_{b, \delta}$ and $L_{b, \delta}$ (respectively, $C_{\sigma,
  \delta}$ and $L_{\sigma, \delta}$) are defined as in
  Lemma~\ref{L:LinearGrowth}, with $f$ replaced by $b$ (respectively,
  $\sigma$). Then, the moment estimates for the spatial and temporal increments
  in part (1) still hold, but with $b(0)$, $\sigma(0)$, $L_b$, and $L_\sigma$
  replaced by $C_{b,\delta}$, $C_{\sigma,\delta}$, $L_{b, \delta}$, and
  $L_{\sigma, \delta}$, respectively.

\end{proposition}
\begin{proof}
  It suffices to prove part (1). Part (2) can be handled in the same way as the
  proof of part (2) in Proposition~\ref{P:MomBdU}. Recall that we use $C$ to
  denote a generic constant that may change its value at each appearance. Let
  {$p \ge 2$ and} $\mathcal{M}(t,p)$ be an upper bound for
  the $L^p(\Omega)$ norm of $u(t,x)$ from~\eqref{E:MomBdU-1}, i.e.,
  \begin{align*}
    \mathcal{M}(t,p) = C \|u_0\|_\infty  \exp\left(C Ht\right).
  \end{align*}
  For any fixed $t > 0$ and $x, x' \in \mathbb{T}$, we have that
  \begin{align*}
    \Norm{u(t,x) - u(t,x')}_p
    \le       & |G_t*u_0(x) - G_t* u_0(x')|                                                                                                   \\
              & + \int_0^t \int_0^1 \left|G_{t-s}(x-y) - G_{t-s}(x'-y)\right| \Norm{b(u(s,y))}_p \ud y  \ud s                                 \\
              & + 2 \sqrt{p} \left(\int_0^t \int_0^1 \left|G_{t-s}(x-y)- G_{t-s}(x'-y)\right|^2 \|\sigma(u(s,y))\|_p\ud y  \ud s\right)^{1/2} \\
    \eqqcolon & \sum_{k=0}^2J_k\,.
  \end{align*}
  For any $\beta \in (0, 1/2\wedge \gamma)$ (this range of $\beta$ will be clear
  from the proof below), we have that
  \begin{align*}
    \MoveEqLeft |x-x'|^{-\beta} \Norm{u(t,x) - u(t,x')}_p \le |x-x'|^{-\beta}(J_0+J_1+J_2)\,.
  \end{align*}
  By Hölder continuity assumption of $u_0$,
  \begin{align*}
    |x-x'|^{-\beta}J_0
    & \le |x-x'|^{-\beta} \int_0^1 G_t(y) |u_0(x-y) - u_0(x'-y)| \ud y \\
    & \le {|u_0|_{\gamma}} \int_0^1 G_t(y) |x-x'|^{\gamma-\beta} \ud y
      \le {|u_0|_{\gamma}}\,,
  \end{align*}
  where we have used the fact that $\beta < \gamma$. For $J_1$, using
  Lemma~\ref{L:Hodler-G}, for any $\beta < 1$, we see that
  \begin{align*}
    |x-x'|^{-\beta}J_1
    \le & \int_0^t \int_0^1 (t-s)^{-\beta/2} \left[G_{2(t-s)}(x-y)+G_{2(t-s)}(x'-y)\right] \left(|b(0)| + L_b \|u(s,y)\|_p \right) \ud y  \ud s \\
    \le & C \int_0^t (t-s)^{-\beta/2}\left(|b(0)| + L_b \mathcal{M}(s,p)\right) \ud s                                                           \\
    \le & C \frac{t^{1-\beta/2}}{1-\beta/2} |b(0)| + C L_b M H^{\beta/2-1} \Gamma(1-\beta/2) \exp \left(C H t\right)\,.
  \end{align*}
  For $J_2$, for any $\beta \in (0, 1/2)$, by applying Lemma~\ref{L:Hodler-G},
  the Burkholder-Davis-Gundy inequality, and the Minkowski inequality, we obtain
  that
  \begin{align*}
       & |x-x'|^{-\beta} J_2                                                                                                                                                              \\
   \le & C \sqrt{p} \left(\int_0^t \int_0^1 (t-s)^{-\beta}\left[G_{2(t-s)}(x-y)+G_{2(t-s)}(x'-y)\right]^2 \left(|\sigma(0)| + L_{\sigma}\|u(s,y)\|_p \right)^2 \ud y  \ud s \right)^{1/2} \\
   \le & C  \sqrt{p} \left(\int_0^t (t-s)^{-\beta}(1+(t-s)^{-1/2})\left(|\sigma(0)| + L_{\sigma}\mathcal{M}(s,p)\right)^2 \ud s \right)^{1/2}                                             \\
   \le & C \sqrt{p} \:|\sigma(0)| \left[\frac{t^{\frac{1}{2}-\beta}}{\frac{1}{2}-\beta} +\frac{t^{1-\beta}}{1-\beta} \right]^{1/2}                                                        \\
       & + C \sqrt{p} L_{\sigma} M  \exp\left( C Ht\right) \left(\Gamma(1/2-\beta)  H^{\beta-1/2} +\Gamma({1}-\beta) H^{\beta-1} \right)^{1/2}.
  \end{align*}
  By combining these three bounds, we see that {for all $\beta\in(0,1/2\wedge\gamma)$}
  \begin{align*}
    \MoveEqLeft \Norm{u(t,x) - u(t,x')}_p \\
    \le
     & C |u_0|_{\gamma} |x-x'|^{\beta}+ C |x-x'|^{\beta} L_b M H^{\beta/2-1} \Gamma(1-\beta/2) \exp \left( C H t\right) \\
     & + C \sqrt{p} |x-x'|^{\beta} L_{\sigma} M \exp \left( C H t\right) \left[\Gamma(1/2-\beta) H^{\beta - 1/2} + \Gamma(1-\beta) H^{\beta-1}\right]^{1/2}                                                   \\
     & + C |x-x'|^\beta|b(0)|\frac{t^{1-\beta/2}}{1-\beta/2} + C\sqrt{p}|x-x'|^\beta|\sigma(0)|\left(\frac{t^{\frac{1}{2}-\beta}}{\frac{1}{2}-\beta}+\frac{t^{1-\beta}}{1-\beta}\right)^{1/2}\,,
  \end{align*}
  which proves~\eqref{E:TimeInc} after simplification. \bigskip

  Similarly, for $0 \le t < t'$,
  \begin{align*}
  \Norm{u(t',x) - u(t,x)}_p 
  \le       & \left|\int_0^1 (G_{t'}(x-y) - G_t(x-y)) u_0(y)\ud y \right|                               \\
            & + \Norm{\int_0^t \int_0^1 [G_{t'-s}(x-y) - G_{t-s}(x-y)]b(u(s,y)) \ud y \ud s}_p          \\
            & + \Norm{\int_t^{t'} \int_0^1 G_{t'-s}(x-y) b(u(s,y))\ud y \ud s}_p                        \\
            & + \Norm{\int_0^t \int_0^1 [G_{t'-s}(x-y) - G_{t-s}(x-y)]\sigma(u(s,y)) W(\ud s,\ud y )}_p \\
            & + \Norm{\int_t^{t'} \int_0^1 G_{t'-s}(x-y) \sigma(u(s,y))W(\ud s,\ud y )}_p               \\
  \eqqcolon & \sum_{i=0}^4 I_i\,.
  \end{align*}
  For $I_0$, from the above proof, we see that $G_tu_0(x) \coloneqq \int_0^1
  G_t(x-y)u_0(y)\ud y$ is Hölder continuous with order $\gamma$, and thus using
  semigroup property of the heat kernel we write
  \begin{align*}
    G_{t'}u_0(x) - G_tu_0(x) = \int_0^1 G_{t'-t}(y)[G_tu_0(x-y)-G_tu_0(x)]\ud y\,,
  \end{align*}
  using the bound
  \begin{equation}
    G_t(x)\le {2\left( 1+\sqrt{\frac{t}{2\pi}} \right)} p_t(x)\,,
  \end{equation}
  where $p_t(x) \coloneqq (2\pi t)^{-1/2}e^{-x^2/{2t}}$ is the heat kernel on
  the real line $\mathbb{R}$ (see, e.g. Lemma 2.1
  in~\cite{chen.ouyang.ea:23:parabolic}), we see that
  \begin{align*}
    I_0
     =  & \left|G_{t'}u_0(x) - G_tu_0(x)\right|
    \le   |u_0|_{\gamma}\int_0^1 G_{t'-t}(y)|y|^{\gamma}\ud y\\
    \le & {2\left( 1+\sqrt{\frac{t'-t}{2\pi}} \right)} |u_0|_{\gamma}\int_0^1 p_{t'-t}(y)|y|^{\gamma}\ud y \\
    \le & {C\left( 1+\sqrt{t'-t} \right)} |u_0|_{\gamma}(t'-t)^{{\gamma/2}}\,.
  \end{align*}
  For $I_1$, using Lemma \ref{L:Hodler-G} we see that
  \begin{align*}
    I_1
    \le & (t'-t)^{\beta/2}\int_0^t (t-s)^{-\beta/2}(|b(0)| + L_b \mathcal{M}(s,p)) \ud s \\
    \le & C(t'-t)^{\beta/2} \frac{t^{1-\beta/2}}{1-\beta/2} |b(0)| + (t'-t)^{\beta/2} C L_b M H^{\beta/2-1} \Gamma(1-\beta/2) \exp \left(C H t\right).
  \end{align*}
  As for $I_2$,
  \begin{align*}
    I_2
    \le & \int_t^{t'}(|b(0)| + L_b \mathcal{M}(s,p)) \ud s                                         \\
    \le & (t'-t) |b(0)| + CL_b M  H^{-1} \left(\exp\left(CHt'\right) - \exp\left(CHt\right)\right) \\
    \le & (t'-t) |b(0)| + C(t'-t)^{\beta/2}L_b M  H^{-1+\beta/2} \exp\left(CHt'\right).
  \end{align*}
  Similarly, for $I_3$ and $I_4$, we have
  \begin{align*}
    I_3
    \le & C \sqrt{p} (t'-t)^{\beta/2} \left(\int_0^t [(t'-s)^{-1/2} +1] (t-s)^{-\beta}\left(|\sigma(0)| + L_{\sigma}\mathcal{M}(s, p)\right)^2 \ud s \right)^{1/2} \\
    \le & C \sqrt{p} (t'-t)^{\beta/2} \bigg(\int_0^t \left[(t-s)^{-\frac{1}{2} -\beta} + (t-s)^{-\beta}\right]|\sigma(0)|^2 \ud s                                  \\
        & \qquad  + \int_0^t \left[(t-s)^{-\frac{1}{2} -\beta} + (t-s)^{-\beta}\right] L_{\sigma}^2 \mathcal{M}(s,p)^2 \ud s \bigg)^{1/2}                          \\
    \le & C \sqrt{p} (t'-t)^{\beta/2} |\sigma(0)| \left[\frac{t^{\frac{1}{2}-\beta}}{\frac{1}{2}-\beta} +\frac{t^{1-\beta}}{1-\beta} \right]^{1/2}                 \\
        & +C \sqrt{p} L_{\sigma} M \exp\left(C H t\right)\left(\Gamma(1/2-\beta)  H^{\beta-1/2} +\Gamma({1}-\beta)  H^{\beta-1} \right)^{1/2}\,,
  \end{align*}
  and
  \begin{align*}
    I_4 & \le C \sqrt{p} \left(\int_t^{t'} ((t'-s)^{-1/2}+1)\left(|\sigma(0)| + L_{\sigma}\mathcal{M}(s, p)\right)^2 \ud s \right)^{1/2} \\
        & \le C \sqrt{p} \left[|\sigma(0)|^2\left((t'-t)^{1/2} + (t'-t)\right) +
              M^2 L_{\sigma}^2\left(\int_t^{t'}(t'-s)^{-1/2}e^{2CHs}\ud s + \int_t^{t'}e^{2CHs}\ud s\right) \right]^{1/2}\,.
  \end{align*}
  Using change of variable we obtain that
  \begin{align*}
    \int_t^{t'}(t'-s)^{-1/2}e^{2CHs}\ud s = e^{2CHt'}(2CH)^{-1/2} \int_0^{2CH(t'-t)} s^{-1/2}e^{-s}\ud s\,,
  \end{align*}
  and the integral can be bounded by both $$ \int_0^{2CH(t'-t)} s^{-1/2}\ud
  s\quad \text{and}\quad \int_0^{\infty} s^{-1/2}e^{-s}\ud s\,, $$ which yields
  that
  \begin{align*}
  \int_t^{t'}(t'-s)^{-1/2}e^{2CHs}\ud s \le & e^{2CHt'}(2CH)^{-\frac12} \min([2CH(t'-t)]^{\frac12}, 1)\\
  \le & e^{2CHt'}(2CH)^{-\frac12+\beta}(t'-t)^{\beta}\,.
  \end{align*}
  In a similar way we get $$ \int_t^{t'}e^{2CHs}\ud s \le (2CH)^{\beta-1}
  e^{2CHt'}(t'-t)^{\beta}\,. $$ Thus we conclude that
  \begin{align*}
    I_4 \le C\sqrt{p} \left[|\sigma(0)|^2((t'-t)^{\frac12}+(t'-t))+M^2 (t'-t)^{\beta} L_{\sigma}^2 e^{2CHt'}[(2CH)^{\beta-\frac12}+(2CH)^{\beta-1}]\right]^{\frac{1}{2}} \,.
  \end{align*}
  Combining all the estimates shows that
  \begin{align*}
          \MoveEqLeft \Norm{u(t,x)- u(t',x)}_p                                                                                                                     \\
    \le & C(t'-t)^{\beta/2} \frac{t^{1-\beta/2}}{1-\beta/2} |b(0)|+ C\left(1+\sqrt{t'-t}\right)(t'-t)^{\gamma/2}|u_0|_{\gamma} + (t'-t)|b(0)|                      \\
        & + C(t'-t)^{\beta/2} \Gamma(1-\beta/2) L_b M H^{\beta/2-1}\exp \left(CHt\right)                                                                           \\
        & + C(t'-t)^{\beta/2} L_b M H^{\beta/2-1}\exp \left(CHt'\right)                                                                                            \\
        & + C \sqrt{p} (t'-t)^{\beta/2} L_{\sigma} M  \exp \left(CHt\right) \times \left(\Gamma(1/2-\beta) H^{\beta-1/2} + \Gamma(1-\beta)H^{\beta-1}\right)^{1/2} \\
        & + C\sqrt{p}(t'-t)^{\beta/2}|\sigma(0)|\left(\frac{t^{\frac{1}{2}-\beta}}{\frac{1}{2}-\beta}+\frac{t^{1-\beta}}{1-\beta}\right)^{1/2}                     \\
        & + C \sqrt{p} \left[(t'-t)^{1/2} + (t'-t)\right]^{1/2} |\sigma(0)|                                                                                        \\
        & + C (t'-t)^{\beta/2}L_{\sigma} M \left(H^{\beta-\frac12}+H^{\beta-1}\right)^{1/2}\exp \left(CHt'\right).
  \end{align*}
  Finally, we can use the fact that $0 \le t'-t < 1$ and $\beta \in (0,
  1/2\wedge \gamma)$ to simplify the above bound to obtain~\eqref{E:TimeInc}.
  This completes the proof of Proposition~\ref{P:Holder}.
\end{proof}

\begin{proposition}\label{P:IdentifySol}
  Let $\{u_n\}_{n\ge 1}$ be a sequence of solutions to~\eqref{E:SHE} with
  respective {nonnegative} initial conditions $u_{0,n}$,
  drift terms $b_n$, and diffusion coefficients $\sigma_n$. Assume that $u_{0,
  n}$ are uniformly bounded, namely,
  \begin{align*}
    \sup_{n\ge 1} \sup_{x\in\mathbb{T}} u_{0,n}(x) < \infty\,.
  \end{align*}
  Suppose $b_n$ and $\sigma_n$ satisfy the conditions in
  Lemma~\ref{L:LinearGrowth} with the same constant $\delta$. Let $L_{b_n,
  \delta}$ and $L_{\sigma_n, \delta}$ be the corresponding Lipschitz constants
  given in Lemma~\ref{L:LinearGrowth}. Assume that for some $L_\delta > 0$,
  \[
    \sup_{n\ge 1} \max\left(L_{b_n, \delta}, L_{\sigma_n, \delta}\right) < L_\delta < \infty\:.
  \]
  If the following limits exist
  \begin{align}
     & u(t,x) \coloneqq \lim_{n \to \infty} u_n(t,x), \quad \text{a.s., for all $t \ge 0$ and $x\in \mathbb{T}$},\label{E:IdentifySol} \\
     & b(z)\coloneqq\lim_{n\to \infty}b_n(z),\quad\text{uniformly for $z \ge 0$}, \label{E:b_n} \\
     & \sigma(z)\coloneqq\lim_{n\to \infty}\sigma_n(z),\quad\text{uniformly for $z \ge 0$},\label{E:sigma_n}
  \end{align}
  and if, in addition,
  \begin{align}\label{E_:UniLp}
    \sup_{n\ge 1} \sup_{t\in[0,T]}\sup_{x\in\mathbb{T}} \E\left(\left|u_n(t,x)\right|^p \right)<\infty\,,\quad\text{for all $p \ge 2$ and $T>0$\,,}
  \end{align}
  then $u$ is also the $L^p(\Omega)$ limit of $u_n$ and solves~\eqref{E:SHE}.
\end{proposition}

Note that condition~\eqref{E:IdentifySol} implies that the initial condition of
$u$, denoted by $u_0$, is equal to $\lim_{n \to \infty}u_{0,n}$.

\begin{proof}
  Notice that condition~\eqref{E:IdentifySol} says that $u(t,x)$ is the almost
  sure limit of $u_n(t,x)$, and thus {is adapted and}
  measurable with respect to $\mathcal{B}((0,\infty)\times \mathbb{T})\times
  \mathcal{F}$.  Hence, both conditions (1) and (2) in
  Definition~\ref{D:Solution} are satisfied. It remains to show that $u$
  satisfies conditions (3) and (4) in Definition~\ref{D:Solution}. \bigskip

  We first prove the following claim: \medskip

  \noindent\textbf{Claim:} Both $b$ and $\sigma$ are continuous functions on
  $[0,\infty)$ that satisfy the conditions in Lemma~\ref{L:LinearGrowth} with
  the same constant $\delta$. Moreover, the corresponding constants $(L_{b,
  \delta}, C_{b, \delta})$ and $(L_{\sigma, \delta}, C_{\sigma, \delta})$
  satisfy the following conditions:
  \begin{gather}\label{E_:Ldelta}
    L_{b, \delta} \vee L_{\sigma, \delta} \le L_\delta < \infty\,, \\
    C_{{b, \delta}} = \lim_{n \to \infty} C_{b_n, \delta}
    \quad \text{and} \quad C_{{\sigma, \delta}} = \lim_{n \to
    \infty} C_{\sigma_n, \delta}. \label{E_:Cb}
  \end{gather}

  It suffices to show the case of $b$. The assumption of uniform convergence
  in~\eqref{E:b_n} implies the continuity of $b$ in $[0,\infty)$. To
  prove~\eqref{E_:Ldelta}, for all $u, v\in[\delta,\infty]$, by the triangle
  inequality,
  \begin{align*}
    \left|b(u)-b(v)\right|
    \le & \left|b(u)-b_n(u)\right|+\left|b_n(u)-b_n(v)\right|+\left|b_n(v)-b(v)\right| \\
    \le & \left|b(u)-b_n(u)\right|+L_\delta|u-v|+\left|b_n(v)-b(v)\right| \to L_\delta\left|u-v\right|\,,
    \quad \text{as $n\to \infty$}\,.
  \end{align*}
  Hence, $b$ is Lipschitz on $[\delta,\infty)$ with the Lipschitz constant
  bounded by $L_\delta$, which proves~\eqref{E_:Ldelta}. As for~\eqref{E_:Cb},
  by the triangle inequality again, we have that
  \[
    \left|C_{b_n,\delta}-C_{b,\delta}\right|
     =   \left|\sup_{y \in [0,\delta]}|b_{n}(y)|-\sup_{y \in [0,\delta]}|b(y)|\right|
     \le \sup_{y \in [0,\delta]} \left||b_{n}(y)|-|b(y)|\right|
     \le \sup_{y \in [0,\delta]} \left|b_{n}(y)-b(y)\right|,
  \]
  where the right-hand converges to zero as $n\to \infty$ due to the uniform
  convergence assumption in~\eqref{E:b_n}. This proves~\eqref{E_:Cb}. Therefore,
  we have proved the claim. \bigskip

  Condition~\eqref{E_:UniLp} implies that $\sup_{n\ge 1}
  \E\left(\left|u_n(t,x)\right|^p \right)<\infty$ for all $p \ge 2$. This,
  together with the Chebyshev inequality, implies that the sequence $u_n(t,x)$
  is uniformly integrable on {$L^p(\Omega)$}. Hence, for all
  $t\in [0,T]$ and $x\in \mathbb{T}$,
  \begin{equation}\label{E_:L2Conv}
    \E (|u(t,x)|^p) = \lim_{n \to \infty} \E (|u_n(t,x)|^p)\,.
  \end{equation}
  Taking supremum over $t\in [0,T]$ and $x\in \mathbb{T}$ in the above limit, we
  can use condition~\eqref{E_:UniLp} to conclude that
  \begin{align}
    \sup_{t\in[0,T]}\sup_{x\in\mathbb{T}} \E\left(|u(t,x)|^p\right) < \infty\,.
  \end{align}
  As a consequence, condition (3) in Definition~\ref{D:Solution} is satisfied. \bigskip

  In order to show that $u$ satisfies condition (4) in
  Definition~\ref{D:Solution}, denote
  \begin{align}\label{E:U}
    \begin{aligned}
    U(t,x) \coloneqq & \int_0^1 G_{t-s}(x-y)u_0(y)\ud y+\int_0^t\int_0^1G_{t-s}(x-y)b(u(s,y))\ud y\ud s \\
                     & +\int_0^t\int_0^1 G_{t-s}(x-y) \sigma(u(s,y))W(\ud s, \ud y)\,.
    \end{aligned}
  \end{align}
  we have that
  \begin{align*}
     \E\left[(u(t,x)-U(t,x))^2\right]
     \le 2\E\left[\left(u(t,x)-u_n(t,x)\right)^2\right] + 2\E\left[\left(u_n(t,x)-U(t,x)\right)^2\right]\,.
  \end{align*}
  By It\^o's isometry,
  \begin{align*}
    \E\left[(u(t,x)-U(t,x))^2\right]
    \le       & 2\E \left[\left(u(t,x)-u_n(t,x)\right)^2\right]                                                                \\
              & + 6  \left(\int_0^1 G_{t}(x-y)\left(u_{0,n}(y)-u_{0}(y)\right) \ud y\right)^2                                  \\
              & + 6t \int_0^t\int_0^1G_{t-s}(x-y)\E\left[\left(b_n(u_n(s,y))-b(u(s,y))\right)^2\right]\ud y\ud s               \\
              & + 6  \int_0^t\int_0^1 G^2_{t-s}(x-y)\E\left[\left(\sigma_n(u_n)(s,y)-\sigma (u(s,y))\right)^2\right]\ud y\ud s \\
    \eqqcolon & I_{1,n}(t,x) + I_{2,n}(t,x) + I_{3,n}(t,x) + I_{4,n}(t,x)\,.
  \end{align*}
  It suffices to show that $\lim_{n\to \infty} I_{i,n} = 0$, for $i = 1, \dots,
  4$. The case $i = 1$ is proved in~\eqref{E_:L2Conv}. As for the case $i = 2$,
  notice that condition~\eqref{E:IdentifySol} applied to $t = 0$ and the uniform
  boundedness of $\{u_{0, n}\}$ together imply that
  \begin{align}\label{E_:u0bdd}
    u_0=\lim_{n \to \infty}u_{0,n} \quad \text{and} \quad
    \sup_{z\in \mathbb{T}} u_0(z)\le\sup_{n\ge 1}\sup_{z\in\mathbb{T}}u_{0,n}(z) < \infty.
  \end{align}
  Hence, an application of the dominated convergence theorem proves the case $i
  = 2$. As for the case $i = 3$, from~\eqref{E:b_n} and~\eqref{E:IdentifySol},
  we have
  \begin{align}\label{E_:bnas}
    b_n(u_n(s,y)) \to b(u(s,y))\,,\quad\text{a.s., for $s \in [0,T]$ and $y \in \mathbb{T}$}\,.
  \end{align}
  Now applying Lemma~\ref{L:LinearGrowth} to both $b_n$ and $b$, we see that,
  for all $p \ge 2$,
  \begin{align*}
    \E  \left( \left|b_n(u_n(s,y))-b(u(s,y))\right|^p\right)
    \le & 2^{p-1}\E\left(\left|b_n(u_n(s,y))\right|^p\right) + 2^{p-1}\E\left(\left|b(u(s,y))\right|^p\right)                                                                     \\
    \le & 4^{p-1}\sup_{n\ge 1}\left(C_{b_n, \delta}^p+C_{b, \delta}^p\right)                                                                                                      \\
        & +4^{p-1}L_\delta^p \sup_{n\ge 1}\sup_{s \in [0,T]} \sup_{y \in \mathbb{T}} \left[ \E\left(\left|u_n(s,y)\right|^p\right) + \E\left(\left|u(s,y)\right|^p\right)\right ] \\
     <  & \infty\,,
  \end{align*}
  which, together with~\eqref{E_:bnas}, implies that
  \begin{align*}
    \lim_{n\to \infty} I_{3,n}(t,x)
    = 6t \int_0^t\int_0^1G_{t-s}(x-y)\E\left[ \lim_{n \to \infty}\left(b_n(u_n(s,y))-b(u(s,y))\right)^2\right]\ud y\ud s
    = 0.
  \end{align*}
  The proof for the case $i = 4$ is similar to that of $i = 3$. Therefore,
  condition (4) in Definition~\ref{D:Solution} is satisfied. This completes the
  proof of Proposition~\ref{P:IdentifySol}.
\end{proof}

\section{Main result} \label{S:MainResults}

In this section, we state and prove our main result, Theorem~\ref{T:SubCRT}.

\begin{theorem}\label{T:SubCRT}
  (1) Let $u_0 \ge 0$ be bounded above and H\"{o}lder continuous with exponent
  $\gamma \in (0, 1)$. Let $b, \sigma : (0,\infty) \to \mathbb{R}$. Assume that
  there exist constants $\delta\in (0, 1)$, $A_1\in(0, 1)$, and $A_2\in (0,
  1/4)$ such that both
  \begin{align}\label{E:A1A2}
    \left|\frac{b(z)}{z}\right|      = O\left(\left(\log\left(1/z\right)\right)^{A_1}\right) \quad \text{and} \quad
    \left|\frac{\sigma(z)}{z}\right| = O\left(\left(\log\left(1/z\right)\right)^{A_2}\right)\,,
  \end{align}
  hold for all $z\in (0,\delta]$, and both $b$ and $\sigma$ are Lipschitz on
  $[\delta, \infty)$. To avoid trivialities, we also assume that $b$ and
  $\sigma$ are not both identically zero. Then there exists a global solution to
  SHE~\eqref{E:SHE} (in the sense of Definition~\ref{D:Solution}), that is
  nonnegative and satisfies
  \begin{align} \label{E:Mmt-SubCRT}
    \sup_{t \in [0,T]} \sup_{x \in \mathbb{T}} \E\left( u(t,x)^p \right) < \infty
    \quad \text{for all $p \ge 2$ and $T>0$}.
  \end{align}
  Moreover, this solution is unique and strictly positive in $C([0,\infty)
  \times \mathbb{T}, \mathbb{R})$ provided that $u_0$ is bounded away from zero,
  i.e.,
  \begin{align}\label{E:u_0>0}
    \inf_{z \in \mathbb{T}}u_0(z) > 0.
  \end{align}

  \noindent (2) Let $b_1$, $b_2$, and $\sigma$ satisfy all conditions in part
  (1), and let $u_{0,1}$ and $u_{0,2}$ be two bounded functions that are
  bounded away from zero. Assume that $b_1 \le b_2$ and $u_{0,1} \le u_{0,2}$.
  Let $u_1$ (resp. $u_2$) be the solution to~\eqref{E:SHE} with drift term $b =
  b_1$ (resp. $b = b_2$). Suppose that $u_1$ and $u_2$ share the same $\sigma$.
  Then with probability one, $u_1(t,x) \le u_2(t,x)$ for all $t > 0$ and $x \in
  \mathbb{T}$.
\end{theorem}

\begin{remark}\label{R:Comparison}
  In part (2) of Theorem~\ref{T:SubCRT} we requires the additional
  condition~\eqref{E:u_0>0}. This is because the uniqueness result in part (1)
  can only be established under this assumption. However, if one considers
  solutions constructed solely via the localization procedure used in the proof
  of part (1) of Theorem~\ref{T:SubCRT}, then the comparison principle stated in
  part (2) can still be established without imposing condition~\eqref{E:u_0>0}.
\end{remark}

\begin{proof}[Proof of Theorem~\ref{T:SubCRT}]
  The proof consists of three steps. The first two steps below will need the
  additional assumption~\eqref{E:u_0>0}. In Step 3, this assumption will be
  removed. \bigskip

  \noindent\textbf{Step 1.~} In this step, we prove part (1) of
  Theorem~\ref{T:SubCRT} under~\eqref{E:u_0>0}. We will divide this step into
  four sub-steps. \bigskip

  \noindent\textbf{Step 1-1.~} In this sub-step, we identify the solution (as
  given by~\eqref{E:Identify} below) via localization procedure under condition
  that $\inf_{z \in \mathbb{T}}u_0(z) \ge 1$. For $\epsilon < \delta$, define
  the modified $b$ and $\sigma$ as follows:
  \begin{equation}\label{E:Eps}
    b_{\epsilon}(z) \coloneqq
    \begin{dcases}
      \frac{b(\epsilon)}{\epsilon} z & 0 < z \le \epsilon\,, \\
      b(z)                           & z > \epsilon\,,
    \end{dcases}
    \quad \text{and} \quad
    \sigma_{\epsilon}(z) \coloneqq
    \begin{dcases}
      \frac{\sigma(\epsilon)}{\epsilon} z & 0 < z \le \epsilon\,, \\
      \sigma(z)                           & z > \epsilon\,.
    \end{dcases}
  \end{equation}
  Thanks to~\eqref{E:A1A2}, when $\epsilon$ is sufficiently small, the growth
  constants (see~\eqref{E:Lb}) of $b_\epsilon$ and $\sigma_\epsilon$ satisfy
  \begin{equation}\label{E:LipEst}
    L_{b_{\epsilon}}      \le C (\log(1/\epsilon))^{A_1} \qquad\text{and}\qquad
    L_{\sigma_{\epsilon}} \le C (\log(1/\epsilon))^{A_2}.
  \end{equation}
  Now consider the following SHE with the modified $b$ and $\sigma$:
  \begin{equation}\label{E:SHEepsilon}
    \frac{\partial u_\epsilon}{\partial t} = \frac{1}{2} \frac{\partial^2 u_\epsilon}{\partial x^2} + b_\epsilon(u_\epsilon) + \sigma_\epsilon(u_\epsilon)\dot{W}\,.
  \end{equation}
  Let $u_\epsilon$ and $\tilde{u}_\epsilon$ be solutions to~\eqref{E:SHEepsilon}
  with initial conditions $u_\epsilon(0,x) \equiv 1$ and
  $\tilde{u}_\epsilon(0,x) = u_0(x)$, respectively.  Because both $b_\epsilon$
  and $\sigma_\epsilon$ are globally Lipschitz, there exist unique solutions
  $u_\epsilon$ and $\tilde{u}_\epsilon$ to~\eqref{E:SHEepsilon}, starting from
  respective initial conditions. The sample path comparison
  theorem~\ref{T:WeakComp} implies that $u_\epsilon \le \tilde{u}_\epsilon$ a.s.
  Define the following stopping times
  \begin{align*}
    \tau_\epsilon         \coloneqq \inf_{s>0}\left\{\inf_{x\in\mathbb{T}} u_\epsilon(s,x)         \le \epsilon\right\} \quad \text{and} \quad
    \tilde{\tau}_\epsilon \coloneqq \inf_{s>0}\left\{\inf_{x\in\mathbb{T}} \tilde{u}_\epsilon(s,x) \le \epsilon\right\}.
  \end{align*}
  Our goal is to show
  \begin{equation}\label{E:ZeroBlowup}
    \lim_{\epsilon\to 0}\tau_\epsilon=\infty\,,\quad \text{a.s.}
  \end{equation}
  Choose $\epsilon(n)=e^{-n}$ for any positive integer $n$, Define stopping
  times $\{T_k\}_{k\geq0}$ recursively by
  \begin{equation*}
    T_0   \coloneqq 0, \quad  \text{and} \quad
    T_{k} \coloneqq \inf_{0\le s\le t}\left\{s>T_{k-1}, \inf_{x\in \mathbb{T}}u_{\epsilon (k)}(s,x)\le e^{-k}\right\}\quad\text{for $k\ge 1$}\,.
  \end{equation*}
  Fix $T>0$. By monotonicity of $\{T_n\}_{n\ge 1}$, the limit
  in~\eqref{E:ZeroBlowup} is equivalent to
  \begin{equation}
    \lim_{n\to \infty}\P\left\{T_n\le T\right\}
    = \lim_{n\to\infty}\P\left\{\inf_{s\in[0,T]}\inf_{x\in\mathbb{T}}u_{\epsilon(n)}(s,x)\le e^{-n}\right\}
    = 0\,. 
  \end{equation}
  To obtain the above limit, let us first consider the following conditional
  probabilities
  \[
    \P\left\{T_{k+1}-T_{k}\le \frac{2T}{n} \middle| \calF_{T_k}\right\}\,,\quad k \geq 0\,.
  \]
  For all $k$ fixed, by the comparison principle (see Theorem~\ref{T:WeakComp})
  and strong Markov property of the solution $u_{\epsilon(n)}$, we have that
  \begin{align*}
        \P\left\{T_{k+1}-T_{k}\le \frac{2T}{n} \middle| \calF_{T_k}\right\}
    \le \P\left\{\inf_{s\in [0,\frac{2T}{n}]}\inf_{x\in \mathbb{T}}U^{(k)}(s,x)\le e^{-(k+1)}\right\}\,,
  \end{align*}
  where $U^{(k)}(s,x)$ solves stochastic heat equation~\eqref{E:SHEepsilon} with
  initial condition $U^{(k)}(0,x)\equiv e^{-k}$. Denote
  $V^{(k)}(s,x)=e^kU^{(k)}(s,x)$. Thus, $V^{(k)}(0,x)\equiv 1$ and
  $V^{(k)}(s,x)$ solves
  \begin{equation}
    \frac{\partial V^{(k)}}{\partial t} = \frac{1}{2} \frac{\partial^2 V^{(k)}}{\partial x^2} + b_k(V^{(k)}) + \sigma_k(V^{(k)})\xi^{(k)}\,,
  \end{equation}
  where $b_k(u) \coloneqq e^{k}b_\epsilon(e^{-k}u)$ and $\sigma_k(u) \coloneqq
  e^{k}\sigma_\epsilon(e^{-k}u)$ with growth constants satisfying
  \[
    L_{b_k} \le {L_{b_\epsilon}} \quad \text{and} \quad
    L_{\sigma_k} \le L_{\sigma_\epsilon}\:.
  \]
  Here, we have used the strong Markov property and $\xi^{(k)}$ are independent
  copies of space-time white noise. Thus, we further get
  \begin{equation}
    \P\left\{T_{k+1}-T_{k} \le \frac{2T}{n} \middle| \calF_{T_k}\right\} \le \P\left\{\sup_{(s,x)\in [0,\frac{2T}{n}]\times \mathbb{T}} \left|V^{(k+1)}(s,x)-V^{(k+1)}(0,x)\right|\ge 1-\frac{1}{e}\right\}\,.\label{E:ProbIneq}
  \end{equation}
  For the sake of conciseness, we temporally write $\Lip_{b_\epsilon}$ as
  $\Lip_{b}$ and write $\Lip_{\sigma_\epsilon}$ as $\Lip_{\sigma}$. Recall the
  notation in~\eqref{E:H&M}, $H = L_b + p^2 L_{\sigma}^4$. By
  {part (1) of} Proposition~\ref{P:Holder}, together
  with $b_{\epsilon}(0) = \sigma_{\epsilon}(0) = 0$, we see that, for all $\beta
  \in (0, 1/2\wedge \gamma)$ and $T \ge t' \ge t \ge 0$,
  \begin{align*}
    \Norm{V^{(k+1)}(t',x)-V^{(k+1)}(t,x)}_p
    \le & C(t'-t)^{\beta/2}L_bH^{\beta/2-1}e^{CHt'} \\
        & + C\sqrt{p}(t'-t)^{\beta/2}L_\sigma  e^{CHt'} \left(H^{\frac{\beta}{2}-\frac{1}{4}}+H^{\frac{\beta}{2}-\frac{1}{2}}\right) \,.
  \end{align*}
  By our assumption on $b$ and $\sigma$, $H$ is away from zero. Hence, the above
  bound can be simplified as
  \begin{align*}
    \Norm{V^{(k+1)}(t',x)-V^{(k+1)}(t,x)}_p
    \le C(t'-t)^{\beta/2}\left(L_bH^{\beta/2-1} + \sqrt{p} L_\sigma H^{(2\beta-1)/4}\right)e^{CHt'}.
  \end{align*}
  Similarly, for all $x,x'\in \mathbb{T}$ and $t>0$, when $H$ is large enough,
  we have that
  \begin{align*}
    \Norm{V^{(k+1)}(t,x')-V^{(k+1)}(t,x)}_p
    \le C|x'-x|^{\beta}\left(L_bH^{\beta/2-1} + \sqrt{p} L_\sigma H^{(2\beta-1)/4}\right)e^{CHt}.
  \end{align*}
  An application of the Kolmogorov continuity theorem (see, e.g.,
  \cite{dalang.khoshnevisan.ea:09:minicourse} or \cite{kunita:90:stochastic}),
  with the space-time metric $\rho\left((t',x'),(t,x)\right) \coloneqq
  |t'-t|^{\beta/2} + |x'-x|^\beta$, indicates that for any $\tau > 0$ and $0 \le
  \eta<(p-2)/p$, we have
  \begin{align*}
    \sup_k\E \left(\sup_{\substack{x, x'\in \mathbb{T} \\ t, t' \in [0,\tau] \\
    (x,t)\neq(x',t')}}\left|\frac{V^{(k+1)}(t',x')-V^{(k+1)}(t,x)}{\rho\left((t',x'),(t,x)\right)^\eta}\right|^p\right)
    \le \left[C \left(L_bH^{\beta/2-1} + \sqrt{p} L_\sigma H^{(2\beta-1)/4}\right)e^{CH\tau}\right]^p\,,
  \end{align*}
  which implies that
  \begin{align*}
    \sup_k \E \left(\sup_{\substack{x\in \mathbb{T} \\ s \in (0,\tau]}}\left|\frac{V^{(k+1)}(s,x)-V^{(k+1)}(0,x)}{s^{(\beta/2)\eta}}\right|^p\right)
    \le \left[C \left(L_bH^{\beta/2-1} + \sqrt{p} L_\sigma H^{(2\beta-1)/4}\right)e^{CH\tau}\right]^p\,.
  \end{align*}
  Thus we have
  \begin{align*}
    I(\tau) \coloneqq & \sup_k \E \left(\sup_{\substack{x\in \mathbb{T} \\ s \in [0,\tau]}}\left|{V^{(k+1)}(s,x)-V^{(k+1)}(0,x)}\right|^p\right) \\
            \le       & \left[C L_b H^{\frac{\beta}{2}-1}e^{CH\tau}+C\sqrt{p}L_{\sigma}e^{C H\tau}H^{\frac{\beta}{2}-\frac{1}{4}}\right]^p \tau^{\frac{\beta\eta p}{2}}\,.
  \end{align*}
  In the following, we assume $n = 2m$. Let $\tau = \frac{2T}{n} = \frac{T}{m}$,
  we obtain that
  \begin{align*}
    I\left(\frac{T}{m}\right)
    \le C^p \left[ L_b H^{\frac{\beta}{2}-1}+ \sqrt{p}L_{\sigma}H^{\frac{\beta}{2}-\frac{1}{4}}\right]^p e^{CpH\frac{T}{m}} \left(\frac{T}{m}\right)^{\frac{\beta\eta p}{2}}\,.
  \end{align*}
  An application of Chebyshev's inequality gives
  \begin{align*}
    \sup_k\P\left\{\sup_{(s,x)\in [0,\frac{T}{m}]\times \mathbb{T}} \left|V^{(k+1)}(s,x)-V^{(k+1)}(0,x)\right|\ge 1-\frac{1}{e}\right\}
    \le \left(1-1/e\right)^{-p} I\left(\frac{T}{m}\right).
  \end{align*}
  Notice that the case $\left\{T_{2m}<T\right\}$ implies that there are at least
  $m$-many distinct $k\in \{0,\dots,2m-1\}$ such that $\left\{T_{k+1}-T_k\le
  T/m\right\}$. To be more precise, we have
  \begin{align*}
    \P \left\{T_{2m}<T\right\}
    \le & \P\left\{\bigcup_{\{k_i\}_{i=1}^m\subset\{1,2,\dots,2m\}}\bigcap_{i=1}^m \left\{T_{k_i+1}-T_{k_i}\le T/m \right\} \right\} \\
    \le & \sum_{\{k_i\}_{i=1}^m\subset\{1,2,\dots,2m\}} \P\left\{\bigcap_{i=1}^m   \left\{T_{k_i+1}-T_{k_i}\le T/m \right\} \right\} \\
      = & \sum_{\{k_i\}_{i=1}^m\subset\{1,2,\dots,2m\}} \E\left\{\prod_{i=1}^m  \P \left\{T_{k_i+1}-T_{k_i}\le T/m  \middle| \calF_{k_i} \right\}\right\} .
  \end{align*}
  Combined with~\eqref{E:ProbIneq}, we obtain that
  \begin{align*}
    \MoveEqLeft
    \P\left\{T_{2m}\le T\right\}
    \le   \binom{2m}{m} \P\left\{\sup_{(s,x)\in [0,\frac{T}{m}]\times \mathbb{T}} \left|V^{(k+1)}(s,x)-V^{(k)}(0,x)\right|\ge 1-\frac{1}{e}\right\}^{m} \\
    \le & \binom{2m}{m} \exp\left(mp \log \left(\left(C L_b H^{\frac{\beta}{2}-1}+C\sqrt{p}L_{\sigma}H^{\frac{\beta}{2}-\frac14}\right)e^{CHT/m}\right)+\frac{\beta \eta p m}{2}\log \frac{T}{m}\right)\,.
  \end{align*}
  By bounding $L_b H^{-1}$ and $\sqrt{p}L_\sigma H^{-1/4}$ by $1$ in the above
  inequality, we obtain that
  \begin{align*}
    \P\left\{T_{2m}\le T\right\}
    \le & \binom{2m}{m} \exp\left\{mp\log \left(CH^{\frac{\beta}{2}}e^{CH\frac{T}{m}}\right)+\frac{\beta \eta p m}{2} \log \frac{T}{m}\right\} \\
    \le & \binom{2m}{m} \exp\left(mp \left(\log C + \frac{\beta}{2} \log H+CH\frac{T}{m}\right)+\frac{\beta \eta p m}{2}\log\frac{T}{m}\right)  \,.
  \end{align*}
  Apply~\eqref{E:LipEst} with $\epsilon = e^{-{2m}}$ to see that $H \le C
  (2m)^{A_1} + C p^2(2m)^{4A_2}$. Noticing that $\log (a+b)\le \log (2(a\vee
  b))=\log 2+(\log a \vee \log b)$ for $a, b > 0$, we see that
  \begin{align*}
     \log H \le& {\log 2+{2\log C}+ \left( A_1{{\log (2m)}}\vee\left({\log p^2}+4 A_2{{\log (2m)}}\right) \right)}
  \end{align*}
  by Stirling's approximation,
  \[
    \binom{2m}{m} \le \frac{2^{2m+1}}{\sqrt{\pi m}} = \exp \left((2m+1)\log 2- \frac{1}{2} \log (\pi m)\right)\,.
  \]
  Thus, we obtain that
  \begin{equation}\label{E:P_Tm<T}
    \begin{aligned}
    \P \left\{T_m \le T\right\}
    \le & \exp \bigg((2m+1)\log 2- \frac{1}{2} \log (\pi m)+ mp \log C                            \\
        & \quad + \frac{mp\beta}{2}\left({\left(1+A_1+4A_2\right)}\log 2+{2\log C} + (A_1 \log m) \vee (\log p^2 + 4A_2 \log m)\right) \\
        & \quad+ C \left(m^{A_1}+ p^2 m^{4A_2}\right) \frac{T}{m} + \frac{\beta \eta p m}{2}\log T - \frac{\beta \eta p m}{2} \log m\bigg)\,.
    \end{aligned}
  \end{equation}

  By our assumption, $A_1<1$ and $4A_2<1$. We can take $p$ large enough such
  that $(A_1\vee 4A_2)<\eta<1-2/p$. Denote $\lambda \coloneqq
  \left(\eta-(A_1\vee 4A_2)\right)>0$. When $m$ is large enough, we see that the
  expression in the exponent of \eqref{E:P_Tm<T} is dominated by
  \begin{align*}
   -\frac{\beta \lambda p m}{4}\log m\,,
  \end{align*}
  which shows that
  \[
    \lim_{m \to \infty}\P\left\{T_{2m}\le T\right\} = 0\,.
  \]
  By monotonicity of $T_n$, $\P(T_\infty > T) = 1$ for any $T > 0$. Our
  goal~\eqref{E:ZeroBlowup} is proved. By comparison theorem with respect to
  initial value, $\tilde{\tau}_\epsilon\geq\tau_\epsilon$ for any $\epsilon>0$.
  Thus, we have
  \begin{equation}\label{E:Tau-Infity}
    \lim_{\epsilon\to 0}\tilde{\tau}_\epsilon=\lim_{\epsilon\to 0}\tau_\epsilon=\infty\,,
  \end{equation}
  almost surely. Recall that $b$ and $\sigma$ agree with $b_\epsilon$ and
  $\sigma_\epsilon$ for $u>\epsilon$, respectively, we can define coherently
  \begin{align}\label{E_:Identify}
    u(t,x) \coloneqq \tilde{u}_\epsilon(t,x),\quad\text{for all $(t,x)\in [0,\tilde{\tau}_\epsilon]\times\mathbb{T}$.}
  \end{align}
  Finally, from~\eqref{E:Tau-Infity} and~\eqref{E_:Identify}, we are able to
  identify $u$ as the following pathwise limit:
  \begin{align}\label{E:Identify}
    u(t,x) = \lim_{\epsilon \downarrow 0} \tilde{u}_\epsilon(t,x), \quad
    \text{a.s., for all $(t,x)\in [0,\infty)\times\mathbb{T}$}.
  \end{align}
  The almost sure path continuity of $u(t,x)$ is inherited from
  $\tilde{u}_{\epsilon}(t,x)$.

  To prove \eqref{E:Mmt-SubCRT}, we notice that as $\epsilon \to 0$, ${\bf
  1}_{\{t > \tau_{\epsilon}\}}$ decreases to 0 a.s., and we have that 
  \begin{align*}
    \mathbb{E} \left[u(t,x)^p\right] 
    = & \lim_{\epsilon \to 0} \left\{
        \E \left[u(t,x)^p {\bf 1}_{\{t < \tau_{\epsilon}\}}\right] 
      + \E \left[u(t,x)^p {\bf 1}_{\{t > \tau_{\epsilon}\}}\right]
      \right\} \\
      = & \lim_{\epsilon \to 0} \E \left[ \tilde{u}_{\epsilon}(t,x)^p {\bf 1}_{\{t < \tau_{\epsilon}\}} \right]
      <\infty,
  \end{align*}
  uniformly for $x\in \mathbb{T}$ and $t \in [0,T]$. This is because by
  \eqref{E:MomBdU-2} in Proposition \ref{P:MomBdU} the moments of
  $\tilde{u}_{\epsilon}(t,x)$ can be uniformly bounded as long as $\epsilon <
  \delta$. \bigskip

  \noindent\textbf{Step 1-2.~} Now we can demonstrate the strict positivity of
  $u$ identified in~\eqref{E:Identify}. Indeed, for any $T>0$, we have
  \begin{align*}
    \MoveEqLeft \P\left(\inf_{t\in [0,T]}\inf_{x\in\mathbb{T}}u(t,x)>0\right)
    =   \P\left(\bigcup_{n\ge 1}\left\{ \inf_{t\in \left[0,\tau_{1/n}\right]} \inf_{x\in\mathbb{T}}u(t,x)>0\right\} \right)\\
    = & \lim_{n\to \infty}\P\left( \inf_{t\in \left[0,\tau_{1/n}\right]} \inf_{x\in\mathbb{T}}u(t,x)>0 \right)
    =   \lim_{n\to \infty}\P\left( \inf_{t\in \left[0,\tau_{1/n}\right]} \inf_{x\in\mathbb{T}}\tilde{u}_{1/n}(t,x)>0 \right) = 1\,,
  \end{align*}
  where the first equality is due to~\eqref{E:Tau-Infity}. \bigskip

  \noindent\textbf{Step 1-3.~} Now we show that $u$ given in~\eqref{E:Identify}
  solves~\eqref{E:SHE} in the sense of Definition~\ref{D:Solution}. This part is
  covered by Proposition~\ref{P:IdentifySol}. We only need to check the uniform
  convergence of $b_\epsilon$ and $\sigma_\epsilon$ to $b$ and $\sigma$
  respectively as $\epsilon \to 0$. Take $b_\epsilon$ for example. For any
  $\epsilon' > 0$, since $\lim_{z\to0} b(z) = 0$, there exists $\delta > 0$ such
  that $\sup_{z\in [0,\delta]}b(z) \le \epsilon'/4$. For any $\epsilon < \delta$
  and $0 \le z \le \epsilon < \delta$, we have
  \begin{align*}
    \left|b_\epsilon(z)-b(z)\right|=\left|b(\epsilon)\frac{z}{\epsilon}-b(z)\right|
    \le \left|b(\epsilon)\right|\left|\frac{z}{\epsilon}-1\right|+\left|b(\epsilon)-b(z)\right|
    \le 4 \times \frac{\epsilon'}{4}\,,
  \end{align*}
   and $b_\epsilon(z)=b(z)$ for $z>\epsilon$. The uniform convergence of
   $b_\epsilon\to b$ follows. The proof for $\sigma_\epsilon$ is the same.
   \bigskip

  \noindent\textbf{Step 1-4.~} Now we prove the uniqueness of the solution $u$
  given in~\eqref{E:Identify}. Let $v\in
  C\left(\left[0,\infty\right)\times\mathbb{T};\R\right)$ be another strict
  positive solution to~\eqref{E:SHE}. Define stopping times as
  \begin{align*}
    \tilde{\tau}_\epsilon'\coloneqq\inf_{t> 0}\left\{\inf_{x\in\mathbb{T}}\left| v(t,x) \right |\le \epsilon\right\}\,.
  \end{align*}
  By the strict positivity of $v$, we have
  \begin{align}\label{E_:v positive}
      \lim_{\epsilon\to 0}\tilde{\tau}_\epsilon'=\infty\,.
  \end{align}
  On $[0, \tilde{\tau}_\epsilon \wedge \tilde{\tau}_\epsilon']$, both $u$ and
  $v$ solves~\eqref{E:SHE} with $b=b_\epsilon$ and $\sigma=\sigma_\epsilon$
  which are globally Lipschitz. The general uniqueness criteria shows
  \begin{align*}
      u(t,x)\mathds{1}_{\left\{t\le \tilde{\tau}_\epsilon\wedge \tilde{\tau}_\epsilon'\right\}}=v(t,x)\mathds{1}_{\left\{t\le \tilde{\tau}_\epsilon\wedge \tilde{\tau}_\epsilon'\right\}}\quad\text{for all}\quad(t,x)\in[0,\infty)\times \mathbb{T}\,.
  \end{align*}
  By taking $\epsilon\to 0$ and combining with~\eqref{E:Tau-Infity}
  and~\eqref{E_:v positive}, the uniqueness of solution to~\eqref{E:SHE} is
  established. \bigskip

  \noindent\textbf{Step 2.~} In this step, we prove part (2) of
  Theorem~\ref{T:SubCRT}. Recall that condition~\eqref{E:u_0>0} holds for both
  $u_{1,0}$ and $u_{2,0}$. Let both $b_\epsilon$ and $\sigma_\epsilon$ be
  defined as in~\eqref{E:Eps} with $\epsilon < \delta < 1$. Define
  \begin{align*}
    \tau_{1,\epsilon} \coloneqq \inf_{s>0}\left\{\inf_{x\in\mathbb{T}} u_{1,\epsilon}(s,x) \le \epsilon\right\} \quad\text{and}\quad
    \tau_{2,\epsilon} \coloneqq \inf_{s>0}\left\{\inf_{x\in\mathbb{T}} u_{2,\epsilon}(s,x) \le \epsilon\right\}\,.
  \end{align*}
  Define $\tau_\epsilon \coloneqq \tau_{1,\epsilon}\wedge\tau_{2,\epsilon}$. By
  the argument in Step~1,
  \begin{gather*}
    u_1(t,x)\mathds{1}_{\{t<\tau_\epsilon\}} = u_{1,\epsilon}(t,x)\mathds{1}_{\{t<\tau_\epsilon\}} \quad \text{and} \quad
    u_2(t,x)\mathds{1}_{\{t<\tau_\epsilon\}} = u_{2,\epsilon}(t,x)\mathds{1}_{\{t<\tau_\epsilon\}}\,.
  \end{gather*}
  Thanks to the comparison principle (see Theorem~\ref{T:WeakComp}), for all
  $\epsilon \in (0, \delta)$, with probability one, $u_{1,\epsilon}(t,x)\le
  u_{2,\epsilon}(t,x)$, for all $(t,x)\in [0,T]\times \mathbb{T}$. Therefore,
  almost surely,
  \begin{equation}\label{E_:UxInd}
    u_1(t,x)\mathds{1}_{\{t<\tau_\epsilon\}}\le u_2(t,x)\mathds{1}_{\{t<\tau_\epsilon\}}\,.
  \end{equation}
  Recall in Step~1 (see~\eqref{E:Tau-Infity}) we have
  \begin{align}\label{E:IndTau}
    \lim_{\epsilon\to 0}\mathds{1}_{\{t<\tau_\epsilon\}} = 1,  \quad \text{a.s.}
  \end{align}
  Finally, taking the limit of $\epsilon \downarrow 0$ on the both sides
  of~\eqref{E_:UxInd} completes the proof of part (2) of Theorem~\ref{T:SubCRT}.
  \bigskip

  \noindent\textbf{Step 3.~} In this step, we prove part (1) of
  Theorem~\ref{T:SubCRT} without the additional assumption~\eqref{E:u_0>0}. For
  each $n\in \mathbb{N}$, define the initial values $u_{0,n}=u_0+\frac{1}{n}$,
  which are bounded away from zero. From Step~1, there exist $u_n$
  solving~\eqref{E:SHE} with the same $b$, $\sigma$ and initial values
  $u_{0,n}$. By Step~2, we can see that $u_n$ is decreasing with respect to $n$.
  Define
  \begin{align}\label{E:u_n->u}
    u(t,x) \coloneqq \lim_{n\to\infty}u_{n}(t,x),\quad \text{a.s.}\,,\forall (t,x)\in [0,T]\times \mathbb{T}\,.
  \end{align}
  Now we can apply Proposition~\ref{P:IdentifySol} to conclude the proof of
  Theorem~\ref{T:SubCRT}.
\end{proof}

By the properties of pointwise limit~\eqref{E:IndTau}, we see that the solution
constructed via localization procedure satisfies the comparison principle as
pointed out in Remark~\ref{R:Comparison}.

\section{Two extensions}\label{S:Extensions}

In this section, we extend the results of Theorem~\ref{T:SubCRT} to two cases:
(1) the drift term $b$ is allowed to be critical, i.e., $A_1=1$ in
\eqref{E:A1A2}; and (2) both $b$ and $\sigma$ are allowed to have superlinear
growth at infinity. The first extension is covered in Theorem~\ref{T:Critical},
while the second one is covered in Theorem~\ref{T:Superlinear}.

\subsection{The limiting case for the drift term $b$} \label{SS:Limiting-b}

In Theorem~\ref{T:SubCRT}, we established the existence of a global solution
when $b(z)$ and $\sigma(z)$ satisfy condition~\eqref{E:A1A2} for positive $z$
near zero, with $A_1<1$ and $A_2<1/4$. In this subsection, we address the
limiting case $A_1 \uparrow 1$, which requires one additional assumption,
stated in~\eqref{E:Liminf} below. {A typical example is
$b(z) = - z \log(1/z)$.} The solution is constructed as a limit of a sequence
$\{ u_{(\alpha)} \}_{0<\alpha<1}$. We continue to use the notation introduced
in the proof of Theorem~\ref{T:SubCRT}.

\begin{theorem}\label{T:Critical}
  Let $u_0\ge 0$ be bounded above and H\"{o}lder continuous with exponent
  $\gamma \in (0, 1)$. Assume that for some $\delta \in (0, 1)$ and some $A_2
  \in (0, 1/4)$, $b$ and $\sigma$ satisfy~\eqref{E:A1A2} with $A_1 = 1$ for all
  $z\in [0,\delta]$, and both $b$ and $\sigma$ are Lipschitz on $[\delta,
  \infty)$. Additionally, assume that for $z\in (0,\delta]$, 
  \begin{align}\label{E:Liminf}
    \frac{|b(\delta)|}{\delta} \geq \frac{|b(z)|}{z}
  \end{align}
  and $b(z)$ does not change sign. Then there exists a global solution to the
  SHE~\eqref{E:SHE}, which is nonnegative and satisfies~\eqref{E:MomBdU-2}.
  Thus,
  \begin{align}\label{E:Mmt-CRT}
    \sup_{t\in [0,T]}\sup_{x\in \mathbb{T}}\E\left[ u(t,x)^p\right]<\infty\,, \quad \text{for all $p\ge 2$ and $T>0$}\,.
  \end{align}
\end{theorem}

\begin{proof}[Proof of Theorem~\ref{T:Critical}]
  Since our primary interest lies in the existence of solutions, we focus
  exclusively on those constructed via the localization and limiting procedures
  outlined in the proof of Theorem~\ref{T:SubCRT}. As noted in
  Remark~\ref{R:Comparison}, the comparison principle remains valid even without
  the additional condition~\eqref{E:u_0>0}. The proof below is divided into two
  steps. \bigskip

  \noindent\textbf{Step 1.~} In this step, we identify a nonnegative solution
  $u(t,x)$ to~\eqref{E:SHE} as the limit of solutions in the case when $0 < A_1
  < 1$; see~\eqref{E:u_alpha} below. Since for $z\in [0,\delta]$, the sign of
  $b(z)$ remains the same. We use $\theta_b$ to denote its sign. Now for each
  $\alpha\in (0,1)$, define
  \begin{align}\label{E:balpha}
    b_{(\alpha)}(z)=\begin{dcases}
      \theta_b \times z \left(\frac{|b(z)|}{z}\right)^{\alpha} \times \left(\frac{|b(\delta)|}{\delta}\right)^{1-\alpha}\,, & 0 < z \le \delta\,, \\
      b(z)\,,                                                                                                                                           & z > \delta\,.
    \end{dcases}
  \end{align}
  Consider~\eqref{E:SHE} with $b$ replaced by $b_{(\alpha)}$. Since $u_0\ge 0$
  is bounded above and is H\"older continuous, we can apply
  Theorem~\ref{T:SubCRT} with $A_1 = \alpha$ to see that there exists a unique
  global solution, which is denoted by $u_{(\alpha)}$. By the comparison
  property (see Remark~\ref{R:Comparison}), $u_{(\alpha)}$ is a monotone
  non-decreasing (resp. non-increasing) sequence with respect to $\alpha$ when
  $\theta_b=1$ (resp. $\theta_b=-1$). Therefore, the following monotone limit is
  well defined
  \begin{equation}\label{E:u_alpha}
    u(t,x) \coloneqq \lim_{\alpha\to 1}u_{(\alpha)}(t,x),\quad \text{a.s. for all $(t,x)\in [0,T]\times \mathbb{T}$}\,.
  \end{equation}
  Since $u_{(\alpha)}$ is nonnegative, $u$ is also nonnegative. \bigskip

  \noindent\textbf{Step 2.~} In this step, we apply
  Proposition~\ref{P:IdentifySol} to show that the process $u$ constructed
  in~\eqref{E:u_alpha} solves~\eqref{E:SHE}. From the proof
  of~\eqref{E:Mmt-SubCRT} above (see the end of step 1-1 in the proof of
  Theorem~\ref{T:SubCRT}), we see that the moments of $u_{\alpha}(t,x)$ can be
  uniformly bounded for $0 < \alpha < 1$, and thus \eqref{E_:UniLp} is
  satisfied (with $n\geq 1$ replaced by $0 < \alpha < 1$).Thus it suffices to
  show that condition~\eqref{E:b_n} of Proposition~\ref{P:IdentifySol} is
  satisfied, namely,
  \begin{align}\label{E:balpha-uni}
    b_{(\alpha)}(z) \to b(z) \quad \text{uniformly for all $z\ge 0$.}
  \end{align}
  Indeed, from~\eqref{E:balpha}, we only need to consider the case $0 < z \le
  \delta$. We will also restrict attention to the case that $\theta_b = 1$; the
  case that $\theta_b=-1$ can be treated similarly. Under these assumptions, we
  have
  \begin{align*}
    \left|b_{(\alpha)}(z) - b(z)\right|
    =   & |z| \left|\left(\frac{b(z)}{z}\right)^{\alpha}\left(\frac{b(\delta)}{\delta}\right)^{1-\alpha}- \frac{b(z)}{z}\right| \\
    \le & |z|\left|\left(\frac{b(z)}{z}\right)^{\alpha} - \frac{b(z)}{z}\right| + |z|\left|\left(\frac{b(z)}{z}\right)^{\alpha}\left(\frac{b(\delta)}{\delta}\right)^{1-\alpha}-\left(\frac{b(z)}{z}\right)^{\alpha}\right| \\
    \eqqcolon & D_1(\alpha,z)   + D_2(\alpha,z)\,.
  \end{align*}
  For $D_1$, given any $\epsilon > 0$, since $b(z)/z \le{C} |\log z|$, it is
  easy to see that there exists $z_0 \in (0, \delta)$ such that uniformly for
  all $z \in (0, z_0)$ and $\alpha \in (0, 1)$,
  \[
    |z|\left|\left(\frac{b(z)}{z}\right)^{\alpha} - \frac{b(z)}{z}\right| \le |z|\left|\left(\frac{b(z)}{z}\right)^{\alpha} + \frac{b(z)}{z}\right|< \frac{\epsilon}{2}\,.
  \]
  For this $z_0$, $\frac{b(z)}{z}$ is bounded above and away from zero on
  $[z_0,\delta]$ and thus there exists $\alpha_0>0$ such that for all $\alpha
  \in (\alpha_0, 1)$,
  \[
    |z|\left|\left(\frac{b(z)}{z}\right)^{\alpha} - \frac{b(z)}{z}\right| < \frac{\epsilon}{2}\,.
  \]
  Hence, we can conclude that $D_1(\alpha, z)$ converges to zero uniformly for
  $z \in (0,\delta]$ as $\alpha \to 1$. For $D_2$, it suffices to notice that
  $|z| \left(\frac{b(z)}{z}\right)^{\alpha}$ is uniformly bounded for $\alpha
  \in (0,1)$ and $z \in (0,\delta]$ and
  $\left(\frac{b(\delta)}{\delta}\right)^{1-\alpha}-1$ goes to 0 as $\alpha \to
  1$. This completes the proof of Theorem~\ref{T:Critical}.
\end{proof}

\subsection{Superlinearity at infinity}\label{SS:Superlin-Infty}

In this subsection, we establish the existence of the solution when $b(x) =
O(x\log x)$ as $x$ approaches both zero and infinity. Our approach follows the
method of~\cite{dalang.khoshnevisan.ea:19:global}, with certain modifications to
deal with the blowup at zero. The key observation is that the blowup of
Lipschitz coefficients for $b$ and $\sigma$ at zero does not affect the moment
growth of the solution.

\begin{theorem}\label{T:Superlinear}
  Let $u_0> 0$ be bounded above and H\"{o}lder continuous with exponent $\gamma
  \in (0, 1)$. Assume that $b$ and $\sigma$ are locally Lipschitz away from zero
  and infinity and the following conditions hold:
  \begin{align}\label{E:SuperLinear-b}
    \left|\frac{b(z)}{z}\right|=\begin{cases}
      O\left(\left(\log (1/z) \right)^{A_1}\right)\,,\quad \text{as}\  z\to 0\,,\\
      O\left(\log z\right)\,,\quad \text{as}\  z\to \infty\,,
  \end{cases}
  \end{align}
  and
  \begin{align}\label{E:SuperLinear-s}
    \left|\frac{\sigma(z)}{z}\right|=\begin{cases}
        O\left(\left(\log (1/z)\right)^{A_2}\right)\,,\quad \text{as}\ z\to 0\,,\\
        O\left(\left(\log z\right)^{\frac{1}{4}}\right)\,,\quad \text{as}\  z\to \infty\,,
    \end{cases}
  \end{align}
  for some $A_1 \in (0,1)$ and $A_2\in(0,1/4)$. Then there exists a unique
  global solution to the stochastic heat equation~\eqref{E:SHE}. 
\end{theorem}

\begin{proof}
  By~\eqref{E:SuperLinear-b} and \eqref{E:SuperLinear-s}, there exists
  $0<\delta<1$ such that both $b$ and $\sigma$ are bounded on $[0,\delta]$ and
  Lipschitz on $[\delta,M]$ for all $M>1$. Define $b^{(M)}(x) \coloneqq
  b(x\wedge M)$ and $\sigma^{(M)} \coloneqq \sigma(x\wedge M)$ for $M > 1$. Let
  $C_b, C_{\sigma}, L_{b^{(M)},\: \delta}, L_{\sigma^{(M)},\: \delta}$ be as
  defined in Proposition \ref{P:MomBdU}, and we can bound $L_{b^{(M)},\:
  \delta}$ by $O(\log M)\vee 1$ and $L_{\sigma^{(M)},\: \delta}$ by
  $O\left(\left(\log M\right)^{\frac{1}{4}}\right)\vee 1$.
  Following~\eqref{E:H&M}, set
  \begin{equation}\label{E:h}
    H(M) \coloneqq L_{b^{(M)},\: \delta}+p^2L_{\sigma^{(M)},\: \delta}^4
    \quad \text{for all $M>0$.}
  \end{equation}
  Then we can apply Theorem~\ref{T:SubCRT} to get the solution, denoted by
  $u^{(M)}$, to~\eqref{E:SHE} with $b^{(M)}$ and $\sigma^{(M)}$, with the
  following moment bounds:
  \begin{align*}
    \sup_{t\in[0,T]}\sup_{x\in \mathbb{T}} \E \left(\left|u^{(M)}(t,x)\right|^p\right)
    \le \left(\|u_0\|_{\infty} + \frac{C_b}{4 L_{b, \delta}} + \frac{C_\sigma}{4L_{\sigma, \delta}}\right)^p e^{CpH(M)T}
    < \infty \,,
  \end{align*}
  for all $p\ge 2$, $M>1$, and $T>0$. By Lemma~\ref{L:LinearGrowth},
  \begin{equation}
    \Norm{b^{(M)}(u^{(M)}(t,x))}_p \le C_b+L_{b^{(M)}, \delta}\: \Norm{u^{(M)}(t,x)}_p
  \end{equation}
  and
  \begin{equation}
    \Norm{\sigma^{(M)}(u^{(M)}(t,x))}_p\le C_\sigma+L_{\sigma^{(M)},\: \delta}\: \Norm{u^{(M)}(t,x)}_p\,.
  \end{equation}
  By part (2) of Proposition \ref{P:Holder}, we have for all $x, x' \in
  \mathbb{T}$, $0 \le t < t' \le {t+1}$ and $p \ge 2$,
  \begin{align*}
     \MoveEqLeft \Norm{u^{(M)}(t,x) - u^{(M)}(t,x')}_p \\
     \le & C_\beta|x-x'|^{\beta} \Bigg(|u_0|_{\gamma}+C_{b^{(M)}} t^{1-\beta/2}+\sqrt{p}C_{\sigma^{(M)}}\left(t^{1/4-\beta/2}+t^{1/2-\beta/2}\right) \\
         & + \left(L_{b^{(M)},\:\delta}+\sqrt{p}L_{\sigma^{(M)},\:\delta}\right)\left(\Norm{u_0}_\infty+C_{b^{(M)}}+C_{\sigma^{(M)}}\right)e^{CH(M)t'}\Bigg)
  \end{align*}
  and
  \begin{align*}
    \MoveEqLeft \Norm{u^{(M)}(t,x) - u^{(M)}(t',x)}_p \\
    \le & C_{\beta}|t-t'|^{\beta/2} \Bigg(|u_0|_{\gamma}+C_{b^{(M)}}\left(1+t^{1-\beta/2}\right)+\sqrt{p}C_{\sigma^{(M)}}\left(1+t^{1/4-\beta/2}+t^{1/2-\beta/2}\right) \\
        & + \left(L_{b^{(M)},\:\delta}+\sqrt{p}L_{\sigma^{(M)},\:\delta}\right)\left(\Norm{u_0}_\infty+C_{b^{(M)}}+C_{\sigma^{(M)}}\right)e^{CH(M)t'}\Bigg)\,.
  \end{align*}
  Thus an application of Kolmogorov continuity theorem indicates for any
  $\epsilon>0$ and $0\le \eta<(p-2)/p$, 
  \begin{multline*}
    \E\left(\sup_{t\in [0,\epsilon]}
    \sup_{x\in\mathbb{T}}\left|u^{(M)}(t,x)-u^{(M)}(0,x)\right|^p\right) \\
    \le C^p\left(1\vee \epsilon\right)^{\frac{\beta\eta
    p}{2}}\Bigg(|u_0|_\gamma+C_b+\sqrt{p}C_\sigma \left(\log
  M+\sqrt{p}\left(\log M\right)^{\frac{1}{4}}\right)e^{C\left(\log M+p^2 \log
M\right) \epsilon}\Bigg)^p\,. 
  \end{multline*}
  Take $M$ sufficiently large and for all $p\ge 2$ fixed,
  \begin{align*}
    \E\left(\sup_{t\in [0,\epsilon]} \sup_{x\in\mathbb{T}}\left|u^{(M)}(t,x)\right|^p\right)
    \le A^p(\epsilon\vee 1)^{\frac{\beta\eta p}{2}}(B+\log M)^pe^{Cp\epsilon\log M}\,,
  \end{align*}
  where $A,B,C$ depend on $p$. Define
  \begin{align*}
    \tau_M^{(1)} \coloneqq \inf\left\{t>0:\: \sup_{x\in\mathbb{T}}\left|u^{(M)}(t,x)\right|>M\right\}\,.
  \end{align*}
  By Chebyshev's inequality,
  \begin{align*}
    \P\left\{\tau_M^{(1)} \le \epsilon\right\} 
    \le \P\left\{\sup_{t\in [0,\epsilon]}\sup_{x\in\mathbb{T}} u^{(M)}(t,x)\geq M\right\} 
    \le A^p(\epsilon\vee 1)^{\frac{\beta\eta p}{2}}(B+\log M)^pM^{p\left(C\epsilon-1\right)} \,.
  \end{align*}
  If we choose $\epsilon = \frac{1}{2}\min\{C^{-1},1\} \eqqcolon t_0$, the right
  hand side converges to 0 as $M \to \infty$. Thus
  \begin{align*}
    \tau_{\infty}^{(1)} \coloneqq \lim_{M\to \infty}\tau_{M}^{(1)}>t_0\,,\quad \text{a.s.}
  \end{align*}
  Now define
  \begin{align*}
    \tau^{(2)}_M \coloneqq \inf\left\{t>t_0 : \sup_{x\in\mathbb{T}}\left|u^{(M)}(t,x)\right|>M\right\}\,.
  \end{align*}
  With a standard argument, we can show that $u^{(M)}$ is a strong Markov
  process. By restarting the equation at $t_0$, we see that $\tau^{(2)}_M-t_0
  \ge t_0$ and thus $\tau_{\infty}^{(2)} \coloneqq \lim_{M\to \infty}
  \tau^{(2)}_M \ge 2t_0$ almost surely. With an induction procedure, we have
  \begin{align*}
    \tau_{\infty}^{(k)} \coloneqq \lim_{M\to \infty} \tau^{(k)}_M \ge kt_0\quad\text{a.s.}
  \end{align*}
  where
  \begin{align*}
    \tau^{(k)}_M \coloneqq \inf\left\{t>(k-1)t_0:\sup_{x\in\mathbb{T}}\left|u^{(M)}(t,x)\right|>M\right\}\,.
  \end{align*}
  Thus, the solution stays finite for all $t< kt_0$ for all $k>0$, which implies
  $\tau_\infty^{(1)}=\tau_\infty^{(k)}=\infty$. This completes the proof of
  Theorem~\ref{T:Superlinear}.
\end{proof}

\begin{remark}
  If we assume that $\left|\frac{b(z)}{z}\right| = O(\log z)$ and
  $\left|\frac{\sigma(z)}{z}\right|= O \left((\log z)^{\frac{1}{4}}\right)$ as
  $z\to \infty$ together with the same assumptions in Theorem \ref{T:Critical},
  using the same argument as before, we can obtain the existence of a mild
  solution to~\eqref{E:SHE}.
\end{remark}

\appendix
\section{Appendix} \label{S:Appendix}

The following theorem is a specific instance of the comparison theorem
from~\cite[Theorem 2.5]{kotelenez:92:comparison}, adapted to our particular
framework.

\begin{theorem}[Weak comparison principle]\label{T:WeakComp}
  Let $u_i$, $i = 1, 2$, be the solutions to~\eqref{E:SHE} with drift terms
  $b_i$ and the same diffusion coefficient $\sigma$. Let the respective
  (deterministic) initial conditions be $u_{0,1}$ and $u_{0,2}$. Assume that
  $\sigma$ and $b_i$ are globally Lipschitz, $b_1(0) = b_2(0) = \sigma(0) = 0$,
  and $u_{0,i} \in L^2(\mathbb{T})$. If $b_1(z) \le b_2(z)$ for all $z \ge 0$
  and $u_{0,1} \le u_{0,2}$, then with probability one, $u_1(t,x) \le u_2(t,x)$
  for all $t \ge 0$ and $x \in \mathbb{T}$.
\end{theorem}

\begin{remark}
  The original conclusion of Theorem 2.5 of~\cite{kotelenez:92:comparison} is
  for all $t\ge 0$, with probability one, $u_1(t,x) \le u_2(t,x)$ for all $x\in
  \mathbb{T}$. Thanks to the H\"older regularity of the solution (Proposition
  \ref{P:Holder}), one can make the probability null set uniform for both $t \ge
  0$ and $x \in \mathbb{T}$ as stated in Theorem~\ref{T:WeakComp} above.
\end{remark}

The following result, concerning the increment of the heat kernel, is used
extensively in the paper.

\begin{lemma}[Lemma 7.1 in \cite{chen.ouyang.ea:23:parabolic}]\label{L:Hodler-G}
  There exists some universal constant $C > 0$ such that for all $\beta \in
  (0,1]$, $x,y \in \mathbb{T}$ and $t'\ge t >0$ such that
  \begin{align}
    & |G_t(x) - G_{t'}(x)| \le C t^{-\beta/2} G_{2t'}( x) (t'-t)^{\beta/2}\,, \\
    & |G_t(x) - G_t   (y)| \le C t^{-\beta/2} (G_{2t}(x) + G_{2t}(y)) |x-y|^{\beta}\,,
  \end{align}
  where $|x-y|$ is understood as the distance on the torus.
\end{lemma}

\bibliographystyle{amsrefs}
\bibliography{All}

\begin{bibdiv}
\begin{biblist}

\bib{bialynicki-birula.mycielski:76:nonlinear}{article}{
      author={Bialynicki-Birula, Iwo},
      author={Mycielski, Jerzy},
       title={Nonlinear wave mechanics},
        date={1976},
        ISSN={0003-4916},
     journal={Annals of Physics},
      volume={100},
      number={1},
       pages={62\ndash 93},
  url={https://www.sciencedirect.com/science/article/pii/0003491676900579},
}

\bib{chen.foondun.ea:23:global}{article}{
      author={Chen, Le},
      author={Foondun, Mohammud},
      author={Huang, Jingyu},
      author={Salins, Michael},
       title={Global solution for superlinear stochastic heat equation on
  $\mathbb{R}^d$ under osgood-type conditions},
        date={2023October},
     journal={Preprint arXiv:2310.02153},
         url={http://arXiv.org/abs/2310.02153},
}

\bib{chen.huang:23:superlinear}{article}{
      author={Chen, Le},
      author={Huang, Jingyu},
       title={Superlinear stochastic heat equation on {$\Bbb{R}^d$}},
        date={2023},
        ISSN={0002-9939},
     journal={Proc. Amer. Math. Soc.},
      volume={151},
      number={9},
       pages={4063\ndash 4078},
         url={https://doi.org/10.1090/proc/16436},
      review={\MR{4607649}},
}

\bib{chen.ouyang.ea:23:parabolic}{article}{
      author={Chen, Le},
      author={Ouyang, Cheng},
      author={Vickery, William},
       title={Parabolic anderson model with colored noise on torus},
        date={2023August},
     journal={Preprint arXiv:2308.10802, to appear in Bernoulli},
         url={http://arXiv.org/abs/2308.10802},
}

\bib{dalang.khoshnevisan.ea:09:minicourse}{book}{
      author={Dalang, Robert},
      author={Khoshnevisan, Davar},
      author={Mueller, Carl},
      author={Nualart, David},
      author={Xiao, Yimin},
       title={A minicourse on stochastic partial differential equations},
      series={Lecture Notes in Mathematics},
   publisher={Springer-Verlag, Berlin},
        date={2009},
      volume={1962},
        ISBN={978-3-540-85993-2},
        note={Held at the University of Utah, Salt Lake City, UT, May 8--19,
  2006, Edited by Khoshnevisan and Firas Rassoul-Agha},
      review={\MR{1500166}},
}

\bib{dalang.khoshnevisan.ea:19:global}{article}{
      author={Dalang, Robert~C.},
      author={Khoshnevisan, Davar},
      author={Zhang, Tusheng},
       title={Global solutions to stochastic reaction-diffusion equations with
  super-linear drift and multiplicative noise},
        date={2019},
        ISSN={0091-1798},
     journal={Ann. Probab.},
      volume={47},
      number={1},
       pages={519\ndash 559},
         url={https://doi.org/10.1214/18-AOP1270},
      review={\MR{3909975}},
}

\bib{fernandez-bonder.groisman:09:time-space}{article}{
      author={Fern\'{a}ndez~Bonder, Julian},
      author={Groisman, Pablo},
       title={Time-space white noise eliminates global solutions in
  reaction-diffusion equations},
        date={2009},
        ISSN={0167-2789},
     journal={Phys. D},
      volume={238},
      number={2},
       pages={209\ndash 215},
         url={https://doi.org/10.1016/j.physd.2008.09.005},
      review={\MR{2516340}},
}

\bib{foondun.khoshnevisan:09:intermittence}{article}{
      author={Foondun, Mohammud},
      author={Khoshnevisan, Davar},
       title={Intermittence and nonlinear parabolic stochastic partial
  differential equations},
        date={2009},
     journal={Electron. J. Probab.},
      volume={14},
       pages={no. 21, 548\ndash 568},
         url={https://doi.org/10.1214/EJP.v14-614},
      review={\MR{2480553}},
}

\bib{foondun.nualart:21:osgood}{article}{
      author={Foondun, Mohammud},
      author={Nualart, Eulalia},
       title={The {O}sgood condition for stochastic partial differential
  equations},
        date={2021},
        ISSN={1350-7265},
     journal={Bernoulli},
      volume={27},
      number={1},
       pages={295\ndash 311},
         url={https://doi.org/10.3150/20-BEJ1240},
      review={\MR{4177371}},
}

\bib{galaktionov.vazquez:02:problem}{incollection}{
      author={Galaktionov, Victor~A.},
      author={V\'{a}zquez, Juan~L.},
       title={The problem of blow-up in nonlinear parabolic equations},
        date={2002},
      volume={8},
       pages={399\ndash 433},
         url={https://doi.org/10.3934/dcds.2002.8.399},
        note={Current developments in partial differential equations (Temuco,
  1999)},
      review={\MR{1897690}},
}

\bib{han.kim.ea:24:on}{article}{
      author={Han, Beom-Seok},
      author={Kim, Kunwoo},
      author={Yi, Jaeyun},
       title={On the support of solutions to nonlinear stochastic heat
  equations},
        date={2024July},
     journal={Preprint arXiv:2407.06827},
         url={http://arXiv.org/abs/2407.06827},
}

\bib{Kavin.Majee:25:levy}{article}{
      author={Kavin, R.},
      author={Majee, Ananta~K.},
       title={L\'evy {D}riven {S}tochastic {H}eat {E}quation with {L}ogarithmic
  {N}onlinearity: {W}ell-{P}osedness and {L}arge {D}eviation {P}rinciple},
        date={2025},
        ISSN={0095-4616,1432-0606},
     journal={Appl. Math. Optim.},
      volume={91},
      number={2},
       pages={Paper No. 51},
         url={https://doi.org/10.1007/s00245-025-10247-5},
      review={\MR{4881026}},
}

\bib{khoshnevisan:14:analysis}{book}{
      author={Khoshnevisan, Davar},
       title={Analysis of stochastic partial differential equations},
      series={CBMS Regional Conference Series in Mathematics},
   publisher={Published for the Conference Board of the Mathematical Sciences,
  Washington, DC; by the American Mathematical Society, Providence, RI},
        date={2014},
      volume={119},
        ISBN={978-1-4704-1547-1},
         url={https://doi.org/10.1090/cbms/119},
      review={\MR{3222416}},
}

\bib{kotelenez:92:comparison}{article}{
      author={Kotelenez, Peter},
       title={Comparison methods for a class of function valued stochastic
  partial differential equations},
        date={1992},
        ISSN={0178-8051},
     journal={Probab. Theory Related Fields},
      volume={93},
      number={1},
       pages={1\ndash 19},
         url={https://doi.org/10.1007/BF01195385},
      review={\MR{1172936}},
}

\bib{kunita:90:stochastic}{book}{
      author={Kunita, Hiroshi},
       title={Stochastic flows and stochastic differential equations},
      series={Cambridge Studies in Advanced Mathematics},
   publisher={Cambridge University Press, Cambridge},
        date={1990},
      volume={24},
        ISBN={0-521-35050-6},
      review={\MR{1070361}},
}

\bib{millet.sanz-sole:21:global}{article}{
      author={Millet, Annie},
      author={Sanz-Sol\'{e}, Marta},
       title={Global solutions to stochastic wave equations with superlinear
  coefficients},
        date={2021},
        ISSN={0304-4149},
     journal={Stochastic Process. Appl.},
      volume={139},
       pages={175\ndash 211},
         url={https://doi.org/10.1016/j.spa.2021.05.002},
      review={\MR{4264843}},
}

\bib{mueller:97:long}{article}{
      author={Mueller, Carl},
       title={Long time existence for the wave equation with a noise term},
        date={1997},
        ISSN={0091-1798},
     journal={Ann. Probab.},
      volume={25},
      number={1},
       pages={133\ndash 151},
         url={https://doi.org/10.1214/aop/1024404282},
      review={\MR{1428503}},
}

\bib{mueller:98:long-time}{article}{
      author={Mueller, Carl},
       title={Long-time existence for signed solutions of the heat equation
  with a noise term},
        date={1998},
        ISSN={0178-8051},
     journal={Probab. Theory Related Fields},
      volume={110},
      number={1},
       pages={51\ndash 68},
         url={https://doi.org/10.1007/s004400050144},
      review={\MR{1602036}},
}

\bib{mueller:00:critical}{article}{
      author={Mueller, Carl},
       title={The critical parameter for the heat equation with a noise term to
  blow up in finite time},
        date={2000},
        ISSN={0091-1798},
     journal={Ann. Probab.},
      volume={28},
      number={4},
       pages={1735\ndash 1746},
         url={https://doi.org/10.1214/aop/1019160505},
      review={\MR{1813841}},
}

\bib{mueller.sowers:93:blowup}{article}{
      author={Mueller, Carl},
      author={Sowers, Richard},
       title={Blowup for the heat equation with a noise term},
        date={1993},
        ISSN={0178-8051},
     journal={Probab. Theory Related Fields},
      volume={97},
      number={3},
       pages={287\ndash 320},
         url={https://doi.org/10.1007/BF01195068},
      review={\MR{1245247}},
}

\bib{mytnik.perkins:11:pathwise}{article}{
      author={Mytnik, Leonid},
      author={Perkins, Edwin},
       title={Pathwise uniqueness for stochastic heat equations with
  {H}\"{o}lder continuous coefficients: the white noise case},
        date={2011},
        ISSN={0178-8051},
     journal={Probab. Theory Related Fields},
      volume={149},
      number={1-2},
       pages={1\ndash 96},
         url={https://doi.org/10.1007/s00440-009-0241-7},
      review={\MR{2773025}},
}

\bib{osgood:98:beweis}{article}{
      author={Osgood, W.~F.},
       title={Beweis der {E}xistenz einer {L}\"{o}sung der
  {D}ifferentialgleichung {$\frac{{dy}}{{dx}} = f\left( {x,y} \right)$} ohne
  {H}inzunahme der {C}auchy-{L}ipschitz'schen {B}edingung},
        date={1898},
        ISSN={1812-8076},
     journal={Monatsh. Math. Phys.},
      volume={9},
      number={1},
       pages={331\ndash 345},
         url={https://doi.org/10.1007/BF01707876},
      review={\MR{1546565}},
}

\bib{pan.shang.ea:25:large}{article}{
      author={Pan, Tianyi},
      author={Shang, Shijie},
      author={Zhang, Tusheng},
       title={Large {D}eviations of {S}tochastic {H}eat {E}quations with
  {L}ogarithmic {N}onlinearity},
        date={2025},
        ISSN={0926-2601,1572-929X},
     journal={Potential Anal.},
      volume={62},
      number={2},
       pages={439\ndash 463},
         url={https://doi.org/10.1007/s11118-024-10142-8},
      review={\MR{4866619}},
}

\bib{rosen.gerald:69:Dilatation}{article}{
      author={Rosen, Gerald},
       title={Dilatation covariance and exact solutions in local relativistic
  field theories},
        date={1969Jul},
     journal={Phys. Rev.},
      volume={183},
       pages={1186\ndash 1188},
         url={https://link.aps.org/doi/10.1103/PhysRev.183.1186},
}

\bib{salins:21:existence}{article}{
      author={Salins, M.},
       title={Existence and uniqueness for the mild solution of the stochastic
  heat equation with non-{L}ipschitz drift on an unbounded spatial domain},
        date={2021},
        ISSN={2194-0401},
     journal={Stoch. Partial Differ. Equ. Anal. Comput.},
      volume={9},
      number={3},
       pages={714\ndash 745},
         url={https://doi.org/10.1007/s40072-020-00182-7},
      review={\MR{4297238}},
}

\bib{salins:22:global}{article}{
      author={Salins, Michael},
       title={Global solutions for the stochastic reaction-diffusion equation
  with super-linear multiplicative noise and strong dissipativity},
        date={2022},
     journal={Electron. J. Probab.},
      volume={27},
       pages={Paper No. 12, 17},
         url={https://doi.org/10.1214/22-ejp740},
      review={\MR{4372099}},
}

\bib{salins:22:global*1}{article}{
      author={Salins, Michael},
       title={Global solutions to the stochastic reaction-diffusion equation
  with superlinear accretive reaction term and superlinear multiplicative noise
  term on a bounded spatial domain},
        date={2022},
        ISSN={0002-9947},
     journal={Trans. Amer. Math. Soc.},
      volume={375},
      number={11},
       pages={8083\ndash 8099},
         url={https://doi.org/10.1090/tran/8763},
      review={\MR{4491446}},
}

\bib{salins:25:solutions}{article}{
      author={Salins, Michael},
       title={Solutions to the stochastic heat equation with polynomially
  growing multiplicative noise do not explode in the critical regime},
        date={2025},
        ISSN={0091-1798,2168-894X},
     journal={Ann. Probab.},
      volume={53},
      number={1},
       pages={223\ndash 238},
         url={https://doi.org/10.1214/24-aop1704},
      review={\MR{4852006}},
}

\bib{shang.zhang:22:stochastic}{article}{
      author={Shang, Shijie},
      author={Zhang, Tusheng},
       title={Stochastic heat equations with logarithmic nonlinearity},
        date={2022},
        ISSN={0022-0396},
     journal={J. Differential Equations},
      volume={313},
       pages={85\ndash 121},
         url={https://doi.org/10.1016/j.jde.2021.12.033},
      review={\MR{4362370}},
}

\bib{walsh:86:introduction}{incollection}{
      author={Walsh, John~B.},
       title={An introduction to stochastic partial differential equations},
        date={1986},
   booktitle={\'{E}cole d'\'{e}t\'{e} de probabilit\'{e}s de {S}aint-{F}lour,
  {XIV}---1984},
      series={Lecture Notes in Math.},
      volume={1180},
   publisher={Springer, Berlin},
       pages={265\ndash 439},
         url={https://doi.org/10.1007/BFb0074920},
      review={\MR{876085}},
}

\end{biblist}
\end{bibdiv}

\noindent\rule{\textwidth}{0.4pt}
\vspace{0.5em}

\noindent\textbf{Author Information}

\vspace{0.5em}
\noindent Le Chen \\
Department of Mathematics and Statistics, Auburn University\\
Auburn, AL 36849, USA\\
Email: \texttt{lzc0090@auburn.edu}

\vspace{1em}
\noindent Jingyu Huang \\
School of Mathematics, University of Birmingham \\
Birmingham, B15 2TT, UK \\
Email: \texttt{j.huang.4@bham.ac.uk}

\vspace{1em}
\noindent Wenxuan Tao \\
School of Mathematics, University of Birmingham \\
Birmingham, B15 2TT, UK  \\
Email: \texttt{wxt399@student.bham.ac.uk}


\end{document}